\documentclass[a4paper,12pt,reqno]{amsart}
\usepackage[T1]{fontenc}
\usepackage[utf8]{inputenc}
\usepackage[libertinus,vvarbb]{newtx}
\usepackage[margin=1in]{geometry}
\usepackage{enumitem}
\usepackage{xparse}
\usepackage{graphicx}
\usepackage[hidelinks]{hyperref}
\usepackage{subfig}

\numberwithin{equation}{section}

\newtheorem{theorem}{Theorem}[section]
\newtheorem{lemma}[theorem]{Lemma}
\newtheorem{corollary}[theorem]{Corollary}

\theoremstyle{definition}
\newtheorem{definition}[theorem]{Definition}

\newtheorem{example}[theorem]{Example}
\newtheorem{remark}[theorem]{Remark}

\newcommand{\C}{\mathbb{C}}
\newcommand{\R}{\mathbb{R}}
\newcommand{\Z}{\mathbb{Z}}
\newcommand{\N}{\mathbb{N}}
\newcommand{\ph}{\varphi}

\newcommand{\ind}{\mathbb{1}}
\renewcommand*{\P}{\mathbb{P}}
\newcommand{\E}{\mathbb{E}}

\renewcommand{\leq}{\leqslant}
\renewcommand{\geq}{\geqslant}

\mathchardef\mathhyphen="2D

\newcommand{\cm}{\mathscr{C}\mspace{-4mu}\mathscr{M}}
\newcommand{\am}{\mathscr{A}\mspace{-4mu}\mathscr{M}}
\newcommand{\amcm}{\am\mspace{-2mu}\mathhyphen\mspace{-1mu}\cm}

\NewDocumentCommand{\formula}{ssom}{%
 \IfBooleanTF{#1}{%
  \IfBooleanTF{#2}{%
   \IfValueTF{#3}%
    {\begin{align}\label{#3}\begin{gathered}#4\end{gathered}\end{align}}%
    {\begin{gather}#4\end{gather}}%
  }{%
   \IfValueTF{#3}%
    {\begin{align}\label{#3}\begin{aligned}#4\end{aligned}\end{align}}%
    {\begin{gather*}#4\end{gather*}}%
  }%
 }{%
  \IfValueTF{#3}%
   {\begin{align}\label{#3}#4\end{align}}%
   {\begin{align*}#4\end{align*}}%
 }%
}

\begin{document}

\title[First passage locations]{First passage locations for two-dimensional lattice random walks and the bell-shape}
\author{Jacek Wszoła}
\thanks{Work supported by the Polish National Science Centre (NCN) grant no.\@ 2023/49/B/ST1/04303}
\address{\normalfont Jacek Wszoła \\ Department of Pure Mathematics \\ Wrocław University of Science and Technology \\ Wybrzeże Wyspiańskiego 27 \\ 50-370 Wrocław, Poland}
\email{jacek.wszola@pwr.edu.pl}
\date{\today}

\begin{abstract}
Let $(X_n,Y_n)$ be a two-dimensional diagonal random walk on the lattice $\Z^2$, with transition probabilities depending only on the position of $Y_n$. In this paper, we study its first passage locations $X(\tau_a)$, where $\tau_a$ is the first time $Y_n$ hits level $a~\in~\Z$. We prove that the probability mass function of appropriately rescaled $X(\tau_a)$ is a convolution of geometric sequences, two-point sequences and an $\amcm$ (absolutely monotone then completely monotone) sequence. In particular, rescaled first passage locations have bell-shaped distributions. In order to prove our results, we introduce and study two new classes of rational functions with alternating zeros or poles. We also prove analogous theorems for standard random walks on the lattice $\Z^2$ and random walks on the honeycomb lattice.
\end{abstract}

\maketitle

\section{Introduction} \label{sec:Introduction}

\subsection{Random walks and passage times} \label{subsec:Discrete-time random walks}

Consider a discrete-time random walk $X_n$ on $\Z = \{\ldots, -2,-1,0,1,2,\ldots\}$, starting from $x \in \Z$, with jumps of $\pm 1$ occurring with probabilities $p_k$ which depend only on the current state $k=X_n$ of the process. In 1991, Bondesson characterized first passage times of such random walks. In his work \cite{bondesson}, he proved that for $a \in \Z$, the random variable $\tau_a = \min\{n \geq 0: X_n = a\}$ has the same distribution as $x-a+2(X+Y)$, where $X$ and $Y$ are independent, the distribution of $X$ is a mixture of geometric distributions, and $Y$ is a sum of geometrically distributed random variables (for details, see \cite{bondesson}).

In the light of the recent work \cite{kw} of Kwaśnicki and the author, Bondesson's result can be rephrased as follows: for a one-dimensional discrete-time random walk, the variable $\frac12 (\tau_a-x+a)$ has a bell-shaped distribution. A nonnegative sequence $(a(k): k \in \Z)$ is said to be bell-shaped when it converges to zero at $\pm \infty$, and the sequence of its iterated differences $\Delta^n a(k)$ changes sign exactly $n$ times for $n = 1,2,\ldots$ Bell-shaped sequences are a discrete analogue of a more common notion of bell-shaped functions, previously studied by Schoenberg, Hirschman, Simon, Kwaśnicki and others (see \cite{hirschman}, \cite{ks}, \cite{simon}, \cite{kwasnicki}). The notion of a bell-shaped sequence was first introduced in \cite{kw}, where one-sided bell-shaped sequences were identified with convolutions of completely monotone sequences and Pólya frequency sequences. This result was later extended to the class of all (two-sided) bell-shaped sequences in \cite{kw2}. Based on this characterization, we can assert that $X+Y$ in Bondesson's decomposition is indeed a one-sided bell-shaped sequence.

In the present work, we extend Bondesson's result to the two-dimensional case and consider a diagonal random walk $(X_n, Y_n)$ on the lattice $\Z^2$, with starting point $(x,y) \in \Z^2$. We assume that the jumps are $(\pm 1, \pm 1)$ and the corresponding transition probabilities depend only on the current level $Y_n$ of the process. The main objective of this paper is to study the distribution of the first passage location, that is, the first coordinate of the process at the first passage time: $X(\tau_a)$, where $\tau_a = \min\{n \geq 0: Y_n = a\}$ and $a \in \Z$ (see Figure~\ref{fig:fpl}). More precisely, we prove that, after an appropriate rescaling, $X(\tau_a)$ follows a bell-shaped distribution. Our findings are summarized in the following result (for the definition of $\amcm$ sequences, see Section \ref{sec:bell-shape}).

\begin{figure}
    \centering
    \includegraphics[width=0.9\linewidth]{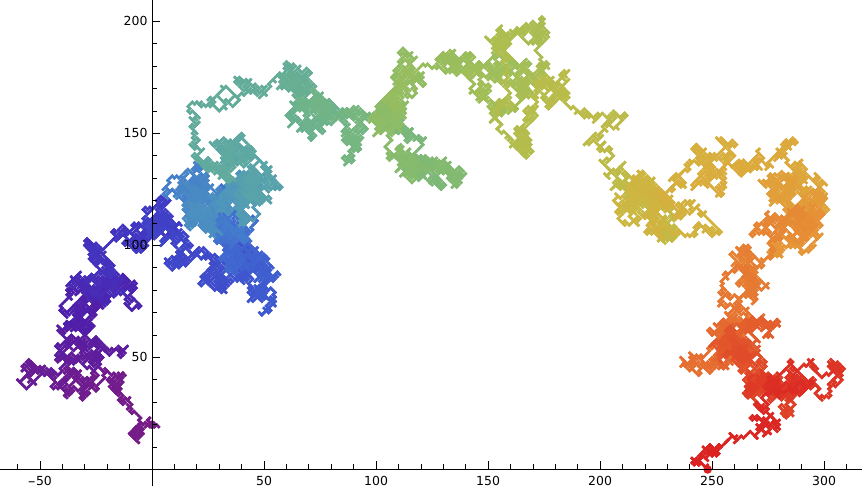}
    \caption{Sample trajectory of a two-dimensional symmetric diagonal random walk $(X_n, Y_n)$ starting from $(0,20)$. Here $X(\tau_0)=248$.}
    \label{fig:fpl}
\end{figure}

\begin{theorem}\label{thm:main}
    Let $(X_n, Y_n)$ be a diagonal random walk that starts in $(x,y) \in \Z^2$, with transition probabilities depending only on $Y_n$. For every $a \in \Z$, $a \neq y$, the first passage location $\frac{1}{2}(X(\tau_a)-|y-a|)-x$ has a bell-shaped probability mass function, given by the convolution of:
    \begin{itemize}
        \item $A$ geometric sequences;
        \item $B$ reflected geometric sequences;
        \item at most $y$ reflected two-point sequences;
        \item an $\amcm$ sequence;
    \end{itemize}
    where $0 \leq A \leq \lfloor \frac{1}{2}(|y-a|-1) \rfloor$ and $0 \leq B \leq \lfloor \frac{1}{2}(|y-a|-1) \rfloor$.
\end{theorem}

Note that Bondesson's theorem is a special case of the above result. It can be easily obtained by applying Theorem \ref{thm:main} to a random walk, which makes jumps only towards the east (NE and SE directions). For details, see Example \ref{ex:bodensson}.

Even though the results presented in this paper may appear similar to those obtained by Bondesson, there are essential differences in the proofs. In both problems, the proof begins from deriving recurrence equations based on the properties of the random walks. The one-dimensional problem yields relatively simple solutions, which can be represented by Stieltjes functions. The two-dimensional case, however, leads to more complex equations and solutions, and hence similar techniques fail to work here. In order to overcome these difficulties, two new classes of rational functions are introduced and studied.

We also prove a completely analogous result for standard random walks on the lattice $\Z^2$. This is the random walk $(X_n, Y_n)$ that makes jumps towards north $(0,1)$, east $(1,0)$, south $(0,-1)$ and west $(-1,0)$, again under the assumption that transition probabilities depend only on the current value of $Y_n$. Like for the diagonal random walk, first passage locations for the standard random walk also follow a bell-shaped distribution.

\begin{theorem}\label{thm:news}
    Let $(X_n, Y_n)$ be the standard random walk that starts in $(x,y) \in \Z^2$, with transition probabilities depending only on $Y_n$. For every $a \in \Z$, $a \neq y$, the first passage location $X(\tau_a)-x$ has a bell-shaped probability mass function given by the convolution of:
    \begin{itemize}
        \item $A$ geometric sequences;
        \item $B$ reflected geometric sequences;
        \item an $\amcm$ sequence;
    \end{itemize}
    where $0 \leq A \leq \lfloor \frac{1}{2}|y-a| \rfloor$ and $0 \leq B \leq \lfloor \frac{1}{2}|y-a| \rfloor$.
\end{theorem}

Although the proof of Theorem~\ref{thm:news} uses similar techniques as the proof of Theorem~\ref{thm:main}, some technical modifications are necessary here; see Section \ref{sec:news}.

Finally, we consider first passage times for a random walk on the honeycomb lattice. This is a two-dimensional random walk that moves only on the vertices of adjacent regular hexagons (for details, see Section \ref{sec:honeycomb}). The problem of first passage times and other properties of random walks on the honeycomb lattice were studied in \cite{dcm}. This type of random walk finds its application in many scientific and engineering problems (see Introduction in \cite{dcm}). It also appeared in the theoretical study of graphene and graphite in \cite{bm}.

Once again we assume that the probabilities of jumps depend only on the position of $Y_n$. As a result of Theorem~\ref{thm:main}, we obtain the following corollary.

\begin{corollary}\label{thm:honey}
    The first passage locations for random walks on the honeycomb lattice satisfy the assertion of Theorem \ref{thm:main}.
\end{corollary}

The results of this work can also find their continuous counterpart. Indeed, in \cite{kwasnicki3} it is proved that Poisson kernels on the half-plane are weakly bell-shaped. 

All random walks discussed in this work are special cases of a broader class of Markov Additive Processes (MAPs). Exit problems of fully discrete MAPs, sometimes called Markov Additive Chains (MACs), were studied in \cite{prp}. In terms of this theory, $X_n$ describes the level of the underlying process while $Y_n$ describes the phase of some external Markov chain; see also \cite{as, ballu, macci}. MACs are essentially the same as hidden Markov models in statistics \cite{elliott, mz}. Both MAPs and MACs are widely used in insurance and risk modeling; see \cite{asmussen1, asmussen2, prp}.

\subsection{Organization} The paper is organized as follows. In Section \ref{sec:generating-functions}, we introduce the assumptions and notation. We also define probability generating functions of the first passage locations and develop recurrence equations connecting them. Section \ref{sec:bell-shape} is devoted to bell-shaped sequences and related classes of sequences. We provide a summary of the history of bell-shaped sequences (and functions) and gather some characterization theorems. The main object of study in Section \ref{sec:rational-functions} are rational functions. We introduce and study two new classes of rational functions: $\mathscr{P}$ and $\mathscr{Q}$. This section is the most technical one: there we prove several useful results regarding $\mathscr{P}$ and $\mathscr{Q}$ functions, including their connections with structures developed in Sections \ref{sec:generating-functions} and  \ref{sec:bell-shape}. The proof of Theorem \ref{thm:main} is provided in Section \ref{sec:proofs}. Finally, in Section \ref{sec:discussion}, we collect some examples and discuss other variants of random walks, which include proving Theorem \ref{thm:news} and Corollary \ref{thm:honey}.

\section{Generating functions of first passage locations}\label{sec:generating-functions}

\subsection{Assumptions and notation} Throughout this paper, our discussion is focused on diagonal random walks. The approach for standard random walks and random walks on the honeycomb lattice is described in detail in Sections \ref{sec:news} and \ref{sec:honeycomb}, respectively.

Let $(X_n, Y_n)$ be a two-dimensional Markov chain on the lattice $\Z^2$. The process only makes diagonal jumps $(\pm 1,\pm 1)$, that is, towards the NE, NW, SE, SW directions. We assume that the transition probabilities depend only on the level $Y_n$ of the process and they are defined as follows:
\formula{
    p_{y,1} = \P(X_{n+1}=X_n+1, Y_{n+1}=y+1 \, | \, X_n = x, Y_n=y), \\
    p_{y,2} = \P(X_{n+1}=X_n-1, Y_{n+1}=y+1 \, | \, X_n = x, Y_n=y), \\
    p_{y,3} = \P(X_{n+1}=X_n-1, Y_{n+1}=y-1 \, | \, X_n = x, Y_n=y), \\
    p_{y,4} = \P(X_{n+1}=X_n+1, Y_{n+1}=y-1 \, | \, X_n = x, Y_n=y),
}
where $\sum_{i=1}^4 p_{y,i} = 1$ for every $y \in \Z$. For simplicity of notation, we denote the jump directions in the same way as the quadrants of the coordinate system. An additional assumption we make is that
\formula{
    & p_{y,1} + p_{y,2} > 0 & \text{and} & & p_{y,3} + p_{y,4} > 0 &
}
for every $y \in \Z$. This simply means that the random walk does not have any natural barriers and it can reach any level $y \in \Z$ with positive probability. This assumption is not essential, as we will discuss later (see Section \ref{sec:barriers}); however, it is convenient to work with it in order to avoid certain technical problems.

The probability measure associated with the random walk starting from $(x,y) \in \Z^2$ will be denoted by $\smash{\P^{(x,y)}}$. An analogous notation will be used for expectations. To avoid nested indices, we will sometimes write $(X(n), Y(n))$ instead of $(X_n, Y_n)$.

The random walk defined above is clearly a time-homogeneous Markov process and it is invariant under translations with respect to the first coordinate $X_n$. The latter property is an easy consequence of the fact that transition probabilities depend only on $Y_n$. This allows us to restrict our attention to processes starting from $(0,y)$.

For $a \in \Z$, we define the first passage time as
\formula{
\tau_a = \min\{n \geq 0: Y(\tau_a) = a\}.
}
Throughout this paper, we will study distributions restricted to the event $\{\tau_a < \infty\}$.

Finally, note that if $|y-a|$ is even, $X(\tau_a)$ is also even. Similarly, if $|y-a|$ is odd, $X(\tau_a)$ is also odd. To ensure that $X(\tau_a)$ can take all integer values, we will consider rescaled $X(\tau_a)$, that is $\frac12 (X(\tau_a) -|y-a|)$. We can assume that $x=a=0$ and $y \in \N = \{0,1,2,\ldots\}$. It follows that the random walk moves only in the lattice $\Z \times \N$. These assumptions will offer major simplifications to the forthcoming proofs.

\subsection{Generating functions}

Our main tool will be the probability generating functions of the first passage locations $X(\tau_a)$. Let us denote
\formula[eq:generating-function]{
f^{a,b}_y(z) = \E^{(0,y)} \left[ z^{X(\tau_a)} \ind_{\tau_a < \tau_b} \right]
}
for $a,b,y \in \Z$. This generating function corresponds to the distributions of the locations of the random walk at the first time of hitting level $a$ from $(0,y)$, given that the process reaches $a$ before hitting $b$. Thus, we additionally assume that $\min\{a,b\} \leq y \leq \max\{a,b\}$. Note that on the event $\{\tau_a < \tau_b\}$, we have $\tau_a < \infty$, and the passage location $X(\tau_a)$ is finite. Therefore, the generating function is well defined on the complex unit circle $\{z \in \C: |z|=1\}$.

Geometric interpretation of formula \eqref{eq:generating-function} described above provides an intuitive tool which can be used to justify the following property of this type of functions.

\begin{lemma}\label{lem:generting-functions-equation}
    If $a < y < b$, the following equation holds:
    \formula*[eq:generting-functions-equation]{ 
        f_y^{a,b}(z) &= (p_{y,3} z^{-1} + p_{y,4} z) f_{y-1}^{a,y}(z) \\ &+\big( (p_{y,1} z + p_{y,2} z^{-1}) \, f_{y+1}^{y, b}(z) + (p_{y,3} z^{-1} + p_{y,4} z) \, f_{y-1}^{y, a}(z) \big) f_y^{a,b}(z).
    }
\end{lemma}

Recall that the function $f_y^{a,b}$ is a sum of expressions of the form $p z^k$ where $k$ is a first passage location and $p$ is a probability of the corresponding trajectory. Every trajectory of random walk which starts from $y$ and reaches $a$ before $b$ can be divided into two distinct groups, depending on the direction of the first jump: up (towards NE or NW) or down (towards SE or SW). If the first jump is up to $y+1$, the process must return to $y$ before reaching $b$. The function $f_{y+1}^{y,b}$ corresponds to this part of the trajectory. On the other hand, if the first jump is down to $y-1$, the processes may either hit $a$ without reaching $y$ again, or it may return to $y$ before hitting $a$. These scenarios are represented by $f_{y-1}^{a,y}$ and $f_{y-1}^{y,a}$, respectively.

\begin{proof}
    We choose integers $a<y<b$ and define $\tau_y^+ = \min\{n > 0: Y_n = y\}$. The expectation in \eqref{eq:generating-function} can be expressed as
    \formula[eq:f-sum-of-expectations]{
    f_y^{a,b}(z) = \E^{(0,y)} \left[ z^{X(\tau_a)} \ind_{\tau_a <\tau_b,\, \tau_a < \tau_y^+} \right] + \E^{(0,y)} \left[ z^{X(\tau_a)} \ind_{\tau_y^+<\tau_a <\tau_b} \right].
    }
    By the Markov property, conditioning the first expectation on $(X_1, Y_1)$ yields
    \formula{
        \E^{(0,y)} \left[ z^{X(\tau_a)} \ind_{\tau_a <\tau_b,\, \tau_a < \tau_y^+} \right] &=(p_{y,3} z^{-1} + p_{y,4} z) \, \E^{(0,y-1)} \left[ z^{X(\tau_a)} \ind_{\tau_a <\tau_b,\, \tau_a < \tau_y} \right] \\ &= (p_{y,3} z^{-1} + p_{y,4} z) \, f_{y-1}^{a,y}(z).
    }
    In order to evaluate the last expectation in \eqref{eq:f-sum-of-expectations}, we apply the strong Markov property with conditioning on $(X(\tau_y^+), Y(\tau_y^+))$. Since $Y(\tau_y^+)=y$, we obtain
    \formula{
        \E^{(0,y)} \left[ z^{X(\tau_a)} \ind_{\tau_y^+<\tau_a <\tau_b} \right] = \E^{(0,y)} \left[ z^{X(\tau_y^+)}\ind_{\tau_y^+ < \tau_a, \tau_y^+ < \tau_b} \right] \E^{(0,y)} \left[ z^{X(\tau_a)} \ind_{ \tau_a <\tau_b}\right].
    }
    It is easy to see that the latter expectation is simply $f_y^{a,b}$, while for the former one we can repeat the arguments used for the first expectation on the right-hand side of \eqref{eq:f-sum-of-expectations}. It follows that
    \formula{
        \E^{(0,y)} \left[ z^{X(\tau_y^+)}\ind_{\tau_y^+ < \tau_a, \tau_y^+ < \tau_b} \right] = (p_{y,1} z + p_{y,2} z^{-1}) \, f_{y+1}^{y, b}(z) + (p_{y,3} z^{-1} + p_{y,4} z) \, f_{y-1}^{y, a}(z)
    }
    Inserting the above results into \eqref{eq:f-sum-of-expectations} yields the assertion of the lemma.
\end{proof}

The lemma might be considered as a first step in proving Theorem \ref{thm:main} because it gives a formula for the generating function of the first passage location. It is expressed in terms of some generating functions which are easier to work with. More precisely, for $a<y<b$ we have
\formula[eq:f-explicit]{
    f_y^{a,b}(z) = \frac{(p_{y,3} z^{-1} + p_{y,4} z) f_{y-1}^{a,y}(z)}{1- (p_{y,1} z + p_{y,2} z^{-1}) \, f_{y+1}^{y, b}(z) - (p_{y,3} z^{-1} + p_{y,4} z) \, f_{y-1}^{y, a}(z) }.
}
If $y=a$ or $y=b$, by \eqref{eq:generating-function}, the generating function is identically equal to 1 or 0, respectively. Similarly, for $b<y<a$, we have
\formula[eq:f-explicit-mirror]{
    f_y^{a,b}(z) = \frac{(p_{y,1} z + p_{y,2} z^{-1}) f_{y+1}^{a,y}(z)}{1- (p_{y,1} z + p_{y,2} z^{-1}) \, f_{y+1}^{y, a}(z) - (p_{y,3} z^{-1} + p_{y,4} z) \, f_{y-1}^{y, b}(z) }.
}

Recall that in \eqref{eq:generating-function}, $b$ denotes an upper barrier of $Y_n$. Since the original random walk has no restrictions on movement in the lattice, we will be finally interested in moving the barrier up to infinity. By taking $b \to \infty$ and using the monotone convergence theorem, we define
\formula[eq:generating-function-limit]{
f_y^{a, \infty}(z) = \lim_{b \to \infty} f_y^{a,b}(z) = \E^{(0,y)} \left[ z^{X(\tau_a)} \ind_{\tau_a < \infty} \right].
}

We will study the location of the first passage of $a=0$ for random walk starting from $(0,y)$, where $y \in \Z$. In this case, by \eqref{eq:f-explicit}, for integer upper barrier $b>0$ we have
\formula[eq:f-explicit-simplified]{
    f_y^{0, b}(z) = \frac{(p_{y,3} z^{-1} + p_{y,4} z) f_{y-1}^{0,y}(z)}{1- (p_{y,1} z + p_{y,2} z^{-1})\, f_{y+1}^{y, b}(z)- (p_{y,3} z^{-1} + p_{y,4} z) \, f_{y-1}^{y, 0}(z) }.
}

\subsection{Recurrences} Now we will focus on functions appearing on the right-hand side of \eqref{eq:f-explicit-simplified}. Observe that with formulae \eqref{eq:f-explicit} and \eqref{eq:f-explicit-mirror} one can immediately construct the recurrence equations for these generating functions.

Substituting $a=0$ and $b=y+1$ into \eqref{eq:f-explicit}, we get
\formula[eq:system-fy0y+1]{
    &f_{y}^{0,y+1}(z) = \frac{(p_{y,3} z^{-1} + p_{y,4} z)f_{y-1}^{0,y}(z)}{1 - (p_{y,3} z^{-1} + p_{y,4} z) f_{y-1}^{y,0}(z)} \quad \text{for } y \geq 1, & f_0^{0,1}(z) = 1. &
}
Similarly, rewriting \eqref{eq:f-explicit-mirror} with $a=y+1$ and $b=0$ yields
\formula[eq:system-fyy+10]{
    &f_{y}^{y+1,0}(z) = \frac{p_{y,1} z + p_{y,2} z^{-1}}{1 - (p_{y,3} z^{-1} + p_{y,4} z) f_{y-1}^{y,0}(z)} \quad \text{for } y \geq 1, & f_0^{1,0}(z) = 0. &
}

We divide both sides of \eqref{eq:system-fy0y+1} and \eqref{eq:system-fyy+10}. Solving for $f_y^{y+1,0}$ yields
\formula[eq:system-fy0y+1-bis]{
f_y^{y+1,0}(z) = \frac{p_{y,1} z + p_{y,2} z^{-1}}{p_{y,3} z^{-1} + p_{y,4} z} \, \frac{f_y^{0, y+1}(z)}{f_{y-1}^{0,y}(z)}.
}
Now substituting this into \eqref{eq:system-fy0y+1}, after some simple manipulations, we have a recurrence given by
\formula{
    \frac{1}{f_y^{0,y+1}(z)} = \frac{1}{p_{y,3} z^{-1} + p_{y,4} z} \, \frac{1}{f_{y-1}^{0,y}(z)} - \frac{p_{y-1,1}z+p_{y-1,2}z^{-1}}{p_{y-1,3}z^{-1}+p_{y-1,4}z} \, \frac{1}{f_{y-2}^{0,y-1}(z)}
}
with initial values
\formula{
& f_0^{0,1}(z) = 1, & f_1^{0,2}(z) = p_{1,3}z^{-1} + p_{1,4}z. &
}

In order to study rescaled passage locations, we introduce another class of probability generating functions, corresponding to the distributions of $\frac{1}{2}(X(\tau_a) - y +a)$. For integers $a,b,y$ such that $\min\{a,b\} \leq y \leq \max\{a,b\}$ and $|w|=1$ denote
\formula[eq:F-definition]{
    F_{y}^{a, b}(w) =  \E^{(0,y) }\left[w^{\frac{1}{2}(X(\tau_a)-y+a)} \ind_{\tau_a < \tau_b}\right].
}
Clearly,
\formula[eq:f-F]{
    f_{y}^{a, b}(z) = z^{y-a} F_{y}^{a, b}(z^2).
}

Using this, one can rewrite \eqref{eq:f-explicit-simplified} in terms of $\smash{F_y^{a,b}}$, with $w=z^2$. The generating function of the rescaled first passage location is thus given by
\formula[eq:F-explicit-simplified]{
F_{y}^{0,b}(w) = \frac{(p_{y,3} w^{-1} + p_{y,4}) F_{y-1}^{0,y}(w)}{1- (p_{y,1} w + p_{y,2}) \, F_{y+1}^{y, b}(w) - (p_{y,3} w^{-1} + p_{y,4}) \, F_{y-1}^{y, 0}(w) }.
}
Moreover, by \eqref{eq:system-fy0y+1-bis}, we have the following recurrence:
\formula[eq:system-Fy0y+1]{
    \frac{1}{F_y^{0,y+1}(w)} = \frac{1}{p_{y,3} w^{-1} + p_{y,4}} \frac{1}{F_{y-1}^{0,y}(w)} - \frac{p_{y-1,1}w+p_{y-1,2}}{p_{y-1,3}w^{-1}+p_{y-1,4}} \frac{1}{F_{y-2}^{0,y-1}(w)},
}
where the initial conditions are
\formula{
   & F_0^{0,1}(w) = 1, & F_1^{0,2}(w) = p_{1,3}w^{-1}+p_{1,4}. &
}
Furthermore, the rescaled counterpart of \eqref{eq:system-fy0y+1-bis} reads
\formula[eq:Fyy+10]{
    F_y^{y+1,0}(w) = \frac{p_{y,1}w+p_{y,2}}{p_{y,3}w^{-1}+p_{y,4}} \, \frac{F_y^{0,y+1}(w)}{F_{y-1}^{0,y}(w)}
}
for $y \geq 1$.

Note that the solution of \eqref{eq:system-Fy0y+1} can be represented as a continued fraction. In the next step, we will prove that it also admits a product representation as a rational function. However, this procedure will require introducing certain special classes of rational functions, which we define and study in Section \ref{sec:rational-functions}. 

\section{Bell-shaped distributions}\label{sec:bell-shape}
\subsection{Bell-shaped sequences} We denote sequences by $a(k)$ instead of $a_k$, where $k$ ranges over the integers. We say that a sequence $a(k)$ is one-sided if $a(k) = 0$ for $k <0$. The $n$th iterated difference of $a(k)$ is given recursively by $\Delta^n a(k) = \Delta \Delta^{n-1} a(k)$, where $\Delta a(k) = a(k+1) - a(k)$. Moreover, a sequence $a(k)$ is said to change sign exactly $N$ times if there is a subsequence $a(k_0), a(k_1), \ldots, a(k_N)$ such that $a(k_{j-1})a(k_j) < 0$ for $j=1,2,\ldots, N$, and $N$ is the biggest number with this property.

We define a convolution of sequences $a(k), b(k)$ in the standard way:
\formula{
    (a*b)(k) = \sum_{j=-\infty}^\infty a(k-j)b(j)
}
whenever the series converges for every $k$. Clearly, this condition is satisfied for probabilistic sequences.

Below we discuss the most important classes of sequences in this paper.

\begin{definition}
    A nonnegative sequence $(a(k): k \in \Z)$, not identically zero, is
    \begin{enumerate}[label = (\alph*)]
        \item bell-shaped if it converges to zero at $\pm \infty$ and the sequence of its $n$th iterated differences $\Delta^n a(k)$ changes sign exactly $n$ times for every $n = 1,2,3,\ldots$;
        \item completely monotone if it is one-sided and the inequality $(-1)^n \Delta^n a(k) \geq 0$ holds for every $n,k =0,1,2\ldots$;
        \item $\amcm$ (absolutely monotone then completely monotone) if sequences $(a(k): k \geq 0)$ and $(a(-k): k \geq 0)$ are both completely monotone;
        \item\label{it:pólya-frequency-definition} summable Pólya frequency sequence if it is summable and for every bounded sequence $b(k)$ the sequence $(a*b)(k)$ has no more sign changes than $b(k)$. 
    \end{enumerate}
\end{definition}

Continuous analogues of bell-shaped sequences, bell-shaped functions, are known in mathematical literature since the 1940s. Although initially studied in the context of statistical games, they can be found throughout the wider field of probability theory. Indeed, many important distributions have bell-shaped densities: normal distribution, Cauchy distribution, or Lévy distribution. In 1984, Gawronski \cite{gawronski} claimed that every stable distribution has a bell-shaped density. His proof, however, turned out to be partially incorrect, and the problem remained open for almost 40 years. The full proof was provided by Kwaśnicki in 2020 \cite{kwasnicki}.

Despite a relatively long history of bell-shaped functions, their discrete counterpart was proposed very recently in \cite{kw}. As it turns out, a lot of results that can be proved for bell-shaped functions hold also in the discrete case. For example, every discrete stable distribution has a bell-shaped probability mass function (see Section 6 in \cite{kw} and Section 1.3 in \cite{kw2}). However, the theory of bell-shaped sequences is not entirely parallel to its continuous counterpart: while there are no compactly supported bell-shaped functions, many finitely supported sequences are bell-shaped.

The class of $\amcm$ sequences, defined in \cite{kw}, was used to characterize all bell-shaped sequences. An analogous class of functions played a role in characterizing all bell-shaped functions in \cite{kwasnicki}. Completely monotone sequences (and functions) have been known and studied for over one hundred years. For example, Hausdorff characterized completely monotone sequences as moment sequences (see \cite{hausdorff1, hausdorff2}). A lot of interesting results are also available for Pólya frequency sequences (and functions), see \cite{karlin}. It is worth mentioning that the original definition is related to the notion of total positivity, while the one presented in item \ref{it:pólya-frequency-definition} is equivalent to it only for summable sequences.

\subsection{Characterization theorems} In this section, we gather some results involving bell-shaped sequences and related classes of sequences. Among the most important tools in this paper, there are characterization theorems. The one for bell-shaped sequences was proved by the author and Kwaśnicki in \cite{kw2} (see Theorem 1.1 therein).

\begin{theorem}\label{thm:bs-characterization}
    A two-sided nonnegative sequence is bell-shaped if and only if it is the convolution of a summable Pólya frequency sequence and an $\amcm$ sequence which converges to zero as $k \to \pm \infty$.
\end{theorem}

The original version of the above theorem provides another equivalent condition: the characteristic function of a bell-shaped sequence admits an exponential representation described in  Theorem 1.1(c) in \cite{kw2}.

The characterization theorem for $\amcm$ sequences is a consequence of the work of Hausdorff. For a detailed proof, see Section 2.3 in \cite{kw2}.

\begin{theorem}\label{thm:amcm-generating-function}
    A sequence $a(k)$ is an $\amcm$ sequence that converges to zero at $\pm \infty$ if and only if there are (positive) measures $\mu_+$ and $\mu_-$ such that
    \formula[eq:amcm-generating-function]{
        \sum_{-\infty}^\infty a(k)x^k =\int_{[0,1)} \frac{1}{1-sx} \, \mu_+(dx) + \int_{[0,1)} \frac{1}{1-s/x} \, \mu_-(dx)-a(0),
    }
    where $|x|=1$, $x \neq 0$ and
    \formula[eq:a(0)]{
    a(0) = \int_{[0,1)} \mu_+(ds) = \int_{[0,1)} \mu_-(ds).
    }
\end{theorem}

Like bell-shaped sequences, generating functions of $\amcm$ sequences also obey an exponential representation described in Lemma 2.2 in \cite{kw2}. The class of generating functions of $\amcm$ sequences is closed under limits, which follows from Lemma 2.5 in \cite{kw2}. The same property is preserved by the class of generating functions of bell-shaped sequences.

The proof of the following characterization theorem for summable Pólya frequency sequences can be found in Karlin's monograph (see Theorem 8.9.5 in \cite{karlin}).

\begin{theorem}\label{thm:pf-generating-function}
    A sequence $a(k)$ is a summable Pólya frequency sequence if and only~if
    \formula{
        \sum_{k = -\infty}^\infty a(k) x^k = x^m \exp\biggl(c^+ x + \frac{c^-}{x} + c\biggr) \prod_{k = 1}^\infty \frac{(1 + p^+_k x)(1 + p^-_k /x)}{(1 - q^+_k x)(1 - q^-_k/x)} \, ,
    }
    where $m$ is an integer, $c^+, c^- \geq0$, $c \in \R$, and $p^\pm_k, q^\pm_k$ are nonnegative summable sequences with $p^\pm_k \le 1$ and $q^\pm_k < 1$.
\end{theorem}

By the form of the generating function of Pólya frequency sequences, one can easily read that this class of sequences contains geometric sequences $a(k) = p^k \ind_{\N}(k)$, two-point sequences $a(k)=p \ind_{\{0\}}(k) + (1-p) \ind_{\{1\}}(k)$, Poisson sequences $a(k) = \lambda^k (k!)^{-1} \ind_{\N}(k)$, their mirror images $a(-k)$, their convolutions, shifts and limits. Here of course $p \in (0,1)$ and $\lambda>0$.

\section{Rational functions}\label{sec:rational-functions}

\subsection{Notation and definitions} Let us consider rational functions $P/Q$, where $Q(x) = x^m$ with $m \in \N$. This type of functions can be identified with generating functions of finitely supported two-sided sequences. In mathematical literature, they are sometimes called Laurent polynomials; see \cite{gm, hendriksen, ss}.

For $n,m \in \N$, let us consider
\formula[eq:lorentz]{
    F(x) = \sum_{k=-m}^n a_k x^k,
}
where $x \in \R$ and $(a_k)$ is a sequence of real numbers such that $a_{n} \neq 0 $ and $a_{-m} \neq 0$. By $d_0(F)$ we denote the number of positive real zeros of $F$, counted with multiplicity. We also define the upper and lower degree of $F$ by, respectively, $d_+(F)=n$ and $d_-(F)=m$.  It is easy to verify that if $F$ and $G$ satisfy \eqref{eq:lorentz}, we have
\formula[]{
    \label{eq:dpm} d_{\pm}(F-G) &\leq \max\{ d_\pm (F), d_\pm (G)\}, \\
    \label{eq:d0} d_{0}(F) &\leq d_+ (F) + d_- (F) .
}
In particular, if $F,G$ are polynomials, we have $d_-(F)=d_-(G)=0$ and the above inequalities are clear.

Sequences (finite or infinite) denoted by Greek letters will be indexed with natural numbers, starting from one. We will say that sequences $(\alpha_k), (\beta_k)$ interlace if $\alpha_k < \beta_k < \alpha_{k+1}$ for every $k = 1,2,3,\ldots$ In order to define the appropriate class of rational functions, we will need a certain class of interlacing sequences, so we will now introduce their formal definition.

\begin{definition}
    We say that (finite) sequences
    \formula{
    & (\alpha_{k}: k = 1,2,\ldots,A), & (\beta_{k}: k = 1,2,\ldots, B), & \\
    & (\alpha'_{k}: k = 1,2,\ldots, A'), & (\beta'_k: k =1,2,\ldots, B') &
    }
    interlace, if $A \leq A' \leq A+1$ and $B \leq B' \leq B+1$, where $A,A', B, B' \in \N$, and the following \textit{interlacing condition} is satisfied:
    \formula[eq:interlacing-condition]{
    0< \ldots < \alpha_{2} < \alpha'_{2} < \alpha_{1} < \alpha'_{1} < 1 < \beta'_{1} < \beta_{1} < \beta'_{2} < \beta_{2} < \ldots
    }
    If $A'=A$, the above sequence of inequalities ends on the left with inequalities $0 < \alpha_A < \alpha'_{A'} < \ldots$ However, if $A' = A+1$, the ending takes the form $0<\alpha'_{A'} < \alpha_{A}<\ldots$ A similar remark applies to the cases $B'=B$ and $B'=B+1$.
\end{definition}

Polynomials with interlacing zeros seem to appear naturally in many different areas of mathematics; see \cite{dj, fisk, gf, johnson}.

\begin{definition}
A function $F$ belongs to $\mathscr{P}$ class if there are $\lambda > 0$ and sequences $(\alpha_{k}:k=1,2,\ldots,A), (\beta_{k}:k=1,2,\ldots,B)$ satisfying
\formula{
0< \alpha_A < \ldots < \alpha_2 <\alpha_1 < 1< \beta_1 < \beta_2< \ldots <\beta_B,
}
such that $F$ can be represented as
\formula[eq:w1-product]{
F(x) = \lambda \prod_{k=1}^{A} \left( \frac{1}{\alpha_{k}} - \frac{1}{x} \right) \prod_{k=1}^{B} (\beta_{k} - x).
}
We will write shortly $F \in \mathscr{P}$.
\end{definition}

In the above definition we have $A = d_-(F)$ and $B = d_+(F)$. Depending on a context, we will sometimes denote $A$ and $B$ from product representation \eqref{eq:w1-product} of $F \in \mathscr{P}$ as $A_F$ and $B_F$, respectively.

For the sake of consistency, we also assume that $\smash{\prod_{k=1}^0 a_k = 1}$ for any sequence $(a_k)$. This implies that every positive constant function is in $\mathscr{P}$, with $A=B=0$. Straight from the definition, it follows that any rational function $F$ satisfying \eqref{eq:lorentz} belongs to class $\mathscr{P}$ if and only if $F(1)>0$ and $d_0(F) = d_+(F) + d_-(F)$.

\begin{definition}
    We say that functions $F, G \in \mathscr{P}$ interlace if there exist positive constants $\lambda, \lambda'$ and interlacing sequences $(\alpha_k), (\beta_k), (\alpha'_k), (\beta'_k)$ such that
    \formula{
        & F(x) = \lambda \prod_{k=1}^{A} \left( \frac{1}{\alpha_{k}} - \frac{1}{x} \right) \prod_{k=1}^{B} (\beta_{k} - x), & G(x) = \lambda' \prod_{k=1}^{A'} \left( \frac{1}{\alpha'_{k}} - \frac{1}{x} \right) \prod_{k=1}^{B'} (\beta'_{k} - x). &
    }
\end{definition}

Note that the interlacing relation is not symmetric and describing it in words can sometimes be confusing, even if in most cases it will be clear from the context. Inspired by Fisk's notation (see \cite{fisk}), if $F$ and $G$ are both in $\mathscr{P}$ and they interlace, we will write shortly $F \ll G$. Following the adopted convention, we assume that if $A=B=0$ ($F$ is constant) and $A' \leq 1, B' \leq 1$, then $F \ll G$. In particular, any two positive constant functions interlace.

\begin{definition}
    A function $F$ belongs to class $\mathscr{Q}$ if there are $\lambda > 0$ and interlacing sequences $(\alpha_{k}), (\beta_k), (\alpha_k'), (\beta_k')$ such that
    \formula{
        F(x) = \lambda \, \cfrac{\displaystyle{\prod_{k=1}^{A} \left( \frac{1}{\alpha_{k}} - \frac{1}{x} \right) \prod_{k=1}^{B} (\beta_{k} - x)}}{\displaystyle{\prod_{k=1}^{A'} \left( \frac{1}{\alpha'_{k}} - \frac{1}{x} \right) \prod_{k=1}^{B'} (\beta'_{k} - x)}}.
    }
    We will write shortly $F \in \mathscr{Q}$.
\end{definition}

Here the name of the class stands for 'zeros-poles interlacing rational functions'. One can immediately deduce that $\mathscr{Q}$ is closely related to interlacing functions. More precisely, $F \in \mathscr{Q}$ if and only if $F=G/H$, where $G \ll H$.

\subsection{Properties of rational functions} Now we prove some properties of functions from classes $\mathscr{P}$ and $\mathscr{Q}$. We present the results in two lemmas below.

\begin{lemma}\label{thm:interlacing-w1}
    Assume that $F \ll G$. Given $a,b,c \geq 0$, not all equal to zero, denote
    \formula{
    H(x) = G(x) - (ax + b + cx^{-1}) \, F(x).
    }
    If $H(1)>0$, then $H \in \mathscr{P}$ and $G \ll H$. Moreover, we have
    \formula{
    & A_H = \max\{A_G, A_F + \ind_{c \neq 0} - \ind_{b=c=0}\}, & B_H = \max\{ B_G, B_F + \ind_{a \neq 0} - \ind_{a =b=0} \}. &
    }
\end{lemma}

\begin{proof}
    Since $F \ll G$, there are some positive constants $\lambda, \lambda'$ and interlacing sequences $(\alpha_k), (\beta_k), (\alpha_k'), (\beta_k')$ such that
    \formula{
        H(x) = \lambda' \prod_{k=1}^{A'} \left( \frac{1}{\alpha'_{k}} - \frac{1}{x} \right) \prod_{k=1}^{B'} (\beta'_{k} - x) - (ax + b + cx^{-1}) \, \lambda \prod_{k=1}^{A} \left( \frac{1}{\alpha_{k}} - \frac{1}{x} \right) \prod_{k=1}^{B} (\beta_{k} - x),
    }
    where $A = A_F, B=B_F$ and $A'=A_G, B'=B_G$. We divide the proof into four steps.

    \textit{Step 1.} Observe that for $l = 1,2,\ldots, A'$ we have
    \formula{
        H(\alpha'_{l}) = - (a\alpha'_{l} + b + c(\alpha'_{l})^{-1})\, \lambda \prod_{k=1}^{A} \left( \frac{1}{\alpha_{k}} - \frac{1}{\alpha'_{l}} \right) \prod_{k=1}^{B} (\beta_{k} - \alpha'_{l}).
    }
    Recall that sequences $(\alpha_k), (\beta_k), (\alpha_k'), (\beta_k')$ satisfy the interlacing condition \eqref{eq:interlacing-condition}. This implies that for all $k,l$ we have
    \formula{
        (a\alpha'_{l} + b + c(\alpha'_{l})^{-1}) \prod_{k=1}^{B} (\beta_{k} - \alpha'_{l}) > 0.
    }
    Furthermore, the expression
    \formula{
        \prod_{k=1}^{A} \left( \frac{1}{\alpha_{k}} - \frac{1}{\alpha'_{l}} \right) = \prod_{k=1}^{l-1} \left( \frac{1}{\alpha_{k}} - \frac{1}{\alpha'_{l}} \right) \prod_{k=l}^{A} \left( \frac{1}{\alpha_{k}} - \frac{1}{\alpha'_{l}} \right)
    }
    has the same sign as $(-1)^{l-1}$. These two facts yield
    \formula{
        &(-1)^l H(\alpha'_{l}) > 0 & \text{for } l = 1,2,\ldots, A'.&
    }
    Using analogous arguments, we deduce that
    \formula{
        &(-1)^l H(\beta'_{l}) > 0 & \text{for } l = 1,2,\ldots, B'. &
    }
    By the intermediate value theorem and the assumption that $H(1)>0$, $H$ has at least $A'+B'$ zeros, which we denote by $(\alpha_k'')$ and $(\beta_k'')$, satisfying the following interlacing condition:
    \formula{
        0<\alpha'_{A'} < \alpha''_{A'} < \ldots < \alpha'_{1}<\alpha''_{1}<1 < \beta''_{1} < \beta'_{1} < \ldots < \beta''_{B'} < \beta'_{B'}.
    }

    \textit{Step 2.} Let us now inspect the upper and lower degrees of $H$. We first find the degrees of its components, that is
    \formula{
        & d_-(G) = A', & d_-((ax+b+cx^{-1}) F) = A + \ind_{c \neq 0} - \ind_{b=c=0}, &\\
        & d_+(G) = B', & d_+((ax+b+cx^{-1}) F) = B + \ind_{a \neq 0} - \ind_{a =b=0}. &
    }
    By \eqref{eq:dpm} and the above, we obtain the upper bounds on the degrees of $H$:
    \formula{
        d_-(H) &\leq \max\{ A', A + \ind_{c \neq 0} - \ind_{b=c=0} \}, \\
        d_+(H) &\leq \max\{ B', B + \ind_{a \neq 0} - \ind_{a =b=0} \}.
    }

    \textit{Step 3.} Suppose that $A' < A + \ind_{c \neq 0} - \ind_{b=c=0}$. Since $A \leq A'$, we have $A' = A$ and $c \neq 0$. Moreover,
    \formula{
        x^{A'+1} H(x) &= \lambda' x \prod_{k=1}^{A'} \left( \frac{x}{\alpha'_{k}} - 1 \right) \prod_{k=1}^{B'} (\beta'_{k} - x) \\ &- (ax^2 + bx + c) \, \lambda \prod_{k=1}^{A'} \left( \frac{x}{\alpha_{k}} - 1 \right) \prod_{k=1}^{B} (\beta_{k} - x).
    }
    Hence, there is $\lambda'' >0$ such that
    \formula{
        \lim_{x \to 0^+} x^{A'+1} H(x) &= (-1)^{A'+1} \lambda''.
    }
    It follows that there exists an additional zero of $H$ which belongs to $(0, \alpha'_{A'})$. We denote it by $\alpha''_{A'+1}$. Therefore, in every case, $H$ has at least $A'' = \max \{A', A + \ind_{c \neq 0} - \ind_{b=c=0}\}$ zeros in the interval $(0, 1)$.

    By the above argument applied to $\smash{\tilde{H}(x)} = H(1/x)$, an analogous claim also holds true: $H$ has at least $B'' = \max\{ B', B + \ind_{a \neq 0} - \ind_{a =b=0} \}$ zeros in the interval $(1,\infty)$.

    \textit{Step 4.} Finally, we can use \eqref{eq:dpm}, \eqref{eq:d0} and the previous step to obtain
    \formula{
        A'' + B'' \leq d_0(H) \leq d_-(H) + d_+(H) \leq A'' + B''.
    }
    Hence, $H$ has exactly $A''$ zeros in $(0,1)$ and exactly $B''$ zeros in $(1, \infty)$. Moreover, we have $d_-(H) = A''$ and $d_+(H) = B''$, which proves that $H \in \mathscr{P}$, and we have already seen that $G \ll H$
\end{proof}

\begin{lemma}\label{thm:w2-closure}
    If functions $F_1, F_2, \ldots, F_n \in \mathscr{Q}$ and nonnegative sequences $(a_k)$, $(b_k)$, $(c_k)$ satisfy
    \formula{
    & \sum_{k=1}^n (a_k + b_k + c_k)>0 & & and  & \sum_{k=1}^n (a_k+b_k+c_k) F_k(1) < 1,&
    }
    then the function defined as
    \formula[eq:H]{
    \frac{1}{1-\sum_{k=1}^n(a_kx+b_k+c_kx^{-1}) F_k(x)}
    }
    is a $\mathscr{Q}$ class function.
\end{lemma}

\begin{proof}
    The proof is again divided into four steps.

    \textit{Step 1.} Given $a,b,c \geq 0$ and $F \in \mathscr{Q}$, denote $W(x) = (ax+b+cx^{-1})F(x)$. We will decompose $W$ into partial fractions. Since $F \in \mathscr{Q}$, there are $\lambda>0$ and interlacing sequences  $(\alpha_k), (\beta_k), (\alpha_k'), (\beta_k')$ such that
    \formula[]{
        W(x) &= \lambda(ax + b +cx^{-1}) \, \cfrac{\displaystyle{\prod_{k=1}^{A} \left( \frac{1}{\alpha_{k}} - \frac{1}{x} \right) \prod_{k=1}^{B} (\beta_{k} - x)}}{\displaystyle{\prod_{k=1}^{A'} \left( \frac{1}{\alpha'_{k}} - \frac{1}{x} \right) \prod_{k=1}^{B'} (\beta'_{k} - x)}} \\
        &= \lambda (ax + b +cx^{-1})x^{A' - A} \, \cfrac{\displaystyle{\prod_{k=1}^{A} \left( x-\alpha_{k} \right) \prod_{k=1}^{B} (\beta_{k} - x)}}{\displaystyle{\prod_{k=1}^{A'} \left( x-\alpha'_{k} \right) \prod_{k=1}^{B'} (\beta'_{k} - x)}}. \label{eq:W}
    }
    Replacing $a,b,c$ with $a/\lambda, b/\lambda, c/\lambda$, we may further assume that $\lambda =1$.

    Now we calculate the coefficients of the partial fractions decomposition. For some $\lambda_j', \lambda_j''>0$ and $j=1,2,\ldots,A'$ we have
    \formula{
        \gamma_{j} = \lim_{x \to \alpha'_{j}} (x-\alpha'_{j}) \, W(x) = \lambda'_j \, \frac{\displaystyle{\prod_{k=1}^{j-1} (\alpha'_{j}-\alpha_{k}) \prod_{k=j}^{A} (\alpha'_{j}-\alpha_{k})}}{\displaystyle{\prod_{k=1}^{j-1} (\alpha'_{j} - \alpha'_{k}) \prod_{k=j+1}^{A'}(\alpha'_{j} - \alpha'_{k})}} = \lambda''_j \, \frac{(-1)^{j-1}}{(-1)^{j-1}} > 0.
    }
    Above, we used inequalities $\alpha'_j < 1<\beta_k$ and $\alpha_j' < 1<\beta'_k$, and the assumption that sequences $(\alpha_k), (\alpha'_{k})$ interlace. By similar arguments, for some other constants $\lambda'_j, \lambda''_j>0$ and $j=1,2,\ldots,B'$ we get
    \formula{
        \delta_{j} = \lim_{x \to \beta'_{j}} (\beta'_{j}-x) \, W(x) = \lambda'_j \, \frac{\displaystyle{\prod_{k=1}^{j-1} (\beta_{k}-\beta'_{j}) \prod_{k=j}^{B} (\beta_{k}-\beta'_{j})}}{\displaystyle{\prod_{k=1}^{j-1} (\beta'_{k} - \beta'_{j}) \prod_{k=j+1}^{B'}(\beta'_{k} - \beta'_{j})}} = \lambda''_j \, \frac{(-1)^{j-1}}{(-1)^{j-1}} > 0.
    }
    Hence, we have the following partial fraction decomposition:
    \formula[eq:W-decomposition]{
        W(x) = \sum_{k=1}^{A'} \frac{\gamma_{k}}{x-\alpha'_{k}} - \sum_{k=1}^{B'} \frac{\delta_{k}}{x-\beta'_{k}} + \frac{\nu}{x} + R(x),
    }
    where $\gamma_j, \delta_j>0$. Here $\nu$ is a real constant and $R$ is a polynomial of degree at most one.

    \textit{Step 2.} By \eqref{eq:W}, $W$ has a simple pole at $x=0$ if and only if $c \neq 0$ and $A = A'$. If there is a pole at $x=0$, then
    \formula{
        \nu = \lim_{x \to 0} x W(x) = \lim_{x \to 0} (ax^2 + bx + c) \, \cfrac{\displaystyle{\prod_{k=1}^{A'} \left( x-\alpha_{k} \right) \prod_{k=1}^{B} (\beta_{k} - x)}}{\displaystyle{\prod_{k=1}^{A'} \left( x-\alpha'_{k} \right) \prod_{k=1}^{B'} (\beta'_{k} - x)}} > 0.
    }
    If there is no pole at $x=0$, then $\nu = 0$. By repeating this argument for $\tilde{W}(x) = W(1/x)$, we assert that the limit
    \formula{
    \mu = \lim_{x \to \infty} \frac{1}{x} W(x) 
    }
    is positive, provided that $B = B'$ and $a \neq 0$. Otherwise, $\mu = 0$.

    Hence, \eqref{eq:W-decomposition} can be rewritten as
    \formula[eq:W-decomposition-2]{
        W(x) = \sum_{k=1}^{A'} \frac{\gamma_{k}}{x-\alpha'_{k}} - \sum_{k=1}^{B'} \frac{\delta_{k}}{x-\beta'_{k}} + \frac{\nu}{x} + \mu x + \eta,
    }
    where $\gamma_k, \delta_k > 0$, $\mu, \nu \geq 0$ and $\eta \in \R$.

    \textit{Step 3.} By the previous step, we deduce that the function $\sum_{k=1}^n (a_k x + b_k + c_k x^{-1}) F_k(x)$, where $(a_k)$, $(b_k)$, $(c_k)$ and $(F_k)$ satisfy the assumptions of the lemma, can be represented as the right hand side of \eqref{eq:W-decomposition-2}. In other words, there are (maybe different from these in the previous step) constants $\gamma_{k}, \delta_{k}~>~0$, $\mu,~\nu~\geq~0$, $\eta~\in~\R$, $A,B \in \mathbb{N}$ and sequences $(\alpha_k), (\beta_k)$ satisfying
    \formula{
    0 < \alpha_A < \ldots < \alpha_2 < \alpha_1 < 1 < \beta_1 < \beta_2 < \ldots < \beta_B,
    }
    such that the denominator of \eqref{eq:H} is equal to
    \formula{
        M(x) = 1-\sum_{k=1}^{A} \frac{\gamma_{k}}{x-\alpha_{k}} + \sum_{k=1}^{B} \frac{\delta_{k}}{x-\beta_{k}} - \frac{\nu}{x} - \mu x - \eta.
    }
    Clearly, the zeros and the poles of the above function are respectively the poles and the zeros of $H$. Thus, it remains to prove that the zeros and the poles of $M$ interlace in a proper way, and the above expression can be represented as in the definition of the class $\mathscr{Q}$.

    \textit{Step 4.} Note that
     \formula{
        M(x) = \frac{P(x)}{x^{\ind_{\nu \neq 0}} \displaystyle{\prod_{k=1}^{A} (x-\alpha_{k}) \prod_{k=1}^{B} (\beta_{k}-x)}},
     }
     where $P$ is a polynomial of degree at most $A+B+\ind_{\nu \neq 0}+\ind_{\mu \neq 0}$. We also have
    \formula{
        & \lim_{x \to \alpha_{k}^\pm} M(x) = \mp \infty, & \lim_{x \to \beta_{k}^\pm} M(x) = \pm \infty&
    }
    and
    \formula{
        & \lim_{x \to \pm \infty}M(x) = \mp \infty \text{ if } \mu \neq 0, &  \lim_{x \to 0^\pm} M(x) = \mp \infty \text{ if } \nu \neq 0.
    }
    Furthermore, the inequality $(a+b+c) \, F(1) + (a'+b'+c') \, G(1) < 1$ implies that $M(1)>0$. By the above properties and the intermediate value theorem, polynomial $P$ has at least $A+B$ zeros, which we denote as $(\alpha_k')$ and $(\beta_k')$, such that
    \formula{
        0<\alpha_A < \alpha'_A < \ldots < \alpha_1 < \alpha'_1 <1< \beta'_1 < \beta_1 < \ldots < \beta'_B < \beta_B.
    }
    
    Observe that if $\mu \neq 0$, then $P$ has at least one additional zero $\alpha'_{A+1} \in (0, \alpha_A)$. If, on the other hand, $\nu \neq 0$, then $P$ has at least one additional zero $\beta'_{B+1} \in (\beta_B, \infty)$. In every case, one can find  at least $A + B +\ind_{\nu \neq 0} + \ind_{\mu \neq 0}$ zeros of $P$, so it has exactly that many zeros. Hence, there are some $\lambda, \lambda' >0$ such that $M$ can be represented as
    \formula{
        M(x) = \lambda \, \frac{\displaystyle{\prod_{k=1}^{A+\ind_{\nu \neq 0}} (x-\alpha'_k) \prod_{k=1}^{B+\ind_{\mu \neq 0}} (\beta'_k-x)}}{x^{\ind_{\nu \neq 0}} \displaystyle{\prod_{k=1}^{A} (x-\alpha_{k}) \prod_{k=1}^{B} (\beta_{k}-x)}} =\lambda' \, \frac{\displaystyle{\prod_{k=1}^{A+\ind_{\nu \neq 0}} \left(\frac{1}{\alpha'_{k}}-\frac{1}{x}\right) \prod_{k=1}^{B+\ind_{\mu \neq 0}} (\beta'_{k}-x)}}{\displaystyle{\prod_{k=1}^{A} \left(\frac{1}{\alpha_{k}}-\frac{1}{x} \right) \prod_{k=1}^{B} (\beta_{k}-x)}}.
    }
    Thus, $H = 1/M \in \mathscr{Q}$, as claimed.
\end{proof}

Finally, we prove two additional results that connect rational functions from classes $\mathscr{P}$ and $\mathscr{Q}$ with generating functions of sequences discussed in the previous section.

\begin{lemma}\label{thm:w1-pf}
    If $F \in \mathscr{P}$, then $1/F$ is the generating function of a summable Pólya frequency sequence.
\end{lemma}

\begin{proof}
    Let $F \in \mathscr{P}$ and observe that
    \formula[eq:w1-pf]{
    F(x) &= \lambda \prod_{k=1}^{A} \left( \frac{1}{\alpha_{k}} - \frac{1}{x} \right) \prod_{k=1}^{B} (\beta_{k} - x) =\lambda' \prod_{k=1}^{A} \left( 1 - \frac{\alpha_k}{x} \right) \prod_{k=1}^{B} \left(1 - \frac{x}{\beta_k}\right)
    }
    where $\lambda'$ is some positive constant. Moreover, $\alpha_j \in (0,1)$ for $j=1,2,\ldots, A$, and since $\beta_j \in (1, \infty)$, we have $1/\beta_j \in (0,1)$ for $j=1,2,\ldots,B$. Thus, the claim is a direct consequence of Theorem~\ref{thm:pf-generating-function}.
\end{proof}

Moreover, by \eqref{eq:w1-pf}, one can see that $F \in \mathscr{P}$ is, up to a multiplication by a constant, the probability generating function of the random variable $\sum_{j=1}^A X_j$ + $\sum_{j=1}^B Y_j$, where $-X_j$ and $Y_j$ are independent, geometrically distributed random variables with parameters $\alpha_j$ and $1/\beta_j$ respectively.

\begin{lemma}\label{thm:amcm-pf}
    If $F \in \mathscr{Q}$, then $F$ is the generating function of an $\amcm$ sequence that converges to zero.
\end{lemma}

\begin{proof}
    We divide the argument into three steps.

    \textit{Step 1.} We first rewrite the condition from Theorem \ref{thm:amcm-generating-function} in a more convenient way. The first integral on the right-hand side of \eqref{eq:amcm-generating-function} reads
    \formula{
        \int_{[0,1)} \frac{1}{1-sx} \mu_+(ds) &= \int_{(0,1)} \frac{1/s}{1/s-x} \mu_+(ds) + \mu_+(\{0\}) \\ &= -\int_{(1,\infty)} \frac{s}{x-s} \tilde{\mu}_+(ds) + \tilde{\mu}_+(\{\infty\}),
    }
    where
    \formula{
        \tilde{\mu}_+(A) = \int_{[0,1)} \ind_{A}(1/s) \, \mu(ds)
    }
    for any Borel set $A$. Above we agree that $1/0 = \infty$. The second integral from \eqref{eq:amcm-generating-function} can be transformed as follows:
    \formula{
        \int_{[0,1)} \frac{1}{1-s/x} \mu_-(ds) &= \int_{(0,1)} \frac{x-s+s}{x-s} \mu_-(ds) + \mu_-(\{0\}) \\ &= \int_{(0,1)} \frac{s}{x-s} \mu_-(ds) + \int_{[0,1)} \mu_-(ds). 
    }
    Hence, \eqref{eq:amcm-generating-function} is equivalent to
    \formula[eq:amcm-generating-function-2]{
        F(w) = \int_{(0,1)} \frac{s}{x-s} \mu_-(ds) - \int_{(1,\infty)} \frac{s}{x-s} \tilde{\mu}_+(ds) + \tilde{\mu}_+(\{\infty\}).
    }
    Moreover, \eqref{eq:a(0)} is equivalent to $\int_{[0,1)} \mu_-(ds) = \int_{(1,\infty]} \tilde{\mu}_+(ds)$.

    \textit{Step 2.} From the first two steps of the proof of Lemma \ref{thm:w2-closure}, with $a=c=0$, it follows that every function $F \in \mathscr{Q}$ can be represented as in \eqref{eq:W-decomposition-2}:
    \formula[eq:F-series]{
        F(x) = \sum_{k=1}^{A} \frac{\gamma_{k}}{x-\alpha_{k}} - \sum_{k=1}^{B} \frac{\delta_{k}}{x-\beta_{k}} + \eta,
    }
    where $\gamma_k, \delta_k>0$, $\eta \in \R$ and sequences $(\alpha_k), (\beta_k)$ satisfy
    \formula{
    0 < \alpha_A < \ldots < \alpha_1 < 1 < \beta_1 < \ldots < \beta_B.
    }
    Here we put $\mu=\nu=0$ because $a=c=0$.

    \textit{Step 3.} We define Borel measures on $[0,1)$ and $(1, \infty]$, respectively:
    \formula{
        & \mu_-(ds) = \sum_{k=1}^{A} \frac{\gamma_{k}}{\alpha_{k}} \varepsilon_{\alpha_{k}}(ds) + \zeta \varepsilon_0(ds), & \tilde{\mu}_+(ds) = \sum_{k=1}^{B} \frac{\delta_{k}}{\beta_{k}} \varepsilon_{\beta_{k}}(ds) +\eta \varepsilon_{\infty}(ds), &
    }
    where $\zeta = F(0)$ and by $\varepsilon_x$ we denote the Dirac measure at $x$. This yields the following representation:
    \formula[eq:F-integral]{
        F(x) = \int_{(0,1)} \frac{s}{x-s} \mu_-(ds) - \int_{(1,\infty)} \frac{s}{x-s} \tilde{\mu}_+(ds) + \tilde{\mu}_+(\{\infty\}).
    }
    Additionally, by \eqref{eq:F-series}, we obtain 
    \formula{
    \zeta = F(0) = -\sum_{k=1}^{A} \frac{\gamma_{k}}{\alpha_{k}} + \sum_{k=1}^{B} \frac{\delta_{k}}{\beta_{k}} + \eta.
    }
    Hence,
    \formula[eq:equal-measures]{
    \int_{[0, 1)} \mu_-(ds) = \sum_{k=1}^{A} \frac{\gamma_k}{\alpha_k} + \zeta = \sum_{k=1}^{B} \frac{\delta_k}{\beta_k} + \eta = \int_{(1,
    \infty]} \tilde{\mu}_+ (ds).
    }
    Comparing \eqref{eq:amcm-generating-function-2} with \eqref{eq:F-integral} and \eqref{eq:equal-measures}, we deduce that $F$ is the generating function of an $\amcm$ sequence, as claimed.
\end{proof}

\section{Proof of the main result}\label{sec:proofs}

\subsection{Distributions of the first passage locations} In the present section, we combine the results obtained in the previous sections and use them to prove Theorem \ref{thm:main}.

Recall that the random walk $(X_n, Y_n)$ starts in $(0,y)$ and we study its location when it hits level zero, that is, $X(\tau_0)$. We have already proved that for $0 < y < b$, equation \eqref{eq:F-explicit-simplified} holds true:
\formula[eq:F-formula]{
    F_{y}^{0,b}(w) = \frac{(p_{y,3} w^{-1} + p_{y,4}) F_{y-1}^{0,y}(w)}{1- (p_{y,1} w + p_{y,2}) \, F_{y+1}^{y, b}(w) - (p_{y,3} w^{-1} + p_{y,4}) \, F_{y-1}^{y, 0}(w) },
}
where $\smash{F_y^{0,y+1}}$ and $\smash{F_y^{y+1,0}}$ are described by equations \eqref{eq:system-Fy0y+1} and \eqref{eq:Fyy+10}. We now solve these equations.

\begin{lemma}\label{thm:Fy0y+1-solution}
    For $y \in \N$ we have
    \formula{
        F_y^{0,y+1}(w) = \frac{1}{G_y(w)} \prod_{k=1}^y (p_{k,3} w^{-1} + p_{k,4}),
    }
    where $G_y \in \mathscr{P}$ and
    \formula[eq:Ay+By]{
    & d_-(G_y) =  A_y \leq \lfloor y/2 \rfloor, & d_+(G_y)= B_y \leq \lfloor y/2 \rfloor. &
    }
    Moreover, $G_y \ll G_{y+1}$.
\end{lemma}

\begin{proof}
    For $y \geq 0$ we define
    \formula{
    G_y(w) = \frac{1}{F_y^{0,y+1}(w)} \prod_{k=1}^y (p_{k,3} w^{-1} + p_{k,4}).
    }
    Rewriting \eqref{eq:system-Fy0y+1} in terms of $G_y$ yields
    \formula[eq:G-recurrence]{
    G_y(w) = G_{y-1}(w) - (p_{y-1,1}w + p_{y-1,2})(p_{y,3}w^{-1} + p_{4,y}) \, G_{y-2}(w)
    }
    with $G_0(w)=G_1(w)=1$. Clearly, $G_0 \ll G_1$. Now we assume that for some $y \geq 0$ we have $G_{y-2} \ll G_{y-1}$. Since $0<p_{y-1,1} + p_{y-1,2}<1$ and $0<p_{y,3}+p_{y,4}<1$, we have
    \formula[eq:abc]{
    (p_{y-1,1}w+p_{y-1,2})(p_{y,3}w^{-1}+p_{y,4}) = aw + b +cw^{-1},
    }
    with $a,b,c \geq 0$, not all equal to zero.
    Moreover, by \eqref{eq:F-definition}, we have
    \formula{
    G_y(1) = \frac{1}{F_y^{y,y+1}(1)} \prod_{k=1}^y (p_{k,3} w^{-1} + p_{k,4}) = \frac{1}{\P^{(0,y)}(\tau_0 < \tau_{y+1})} \prod_{k=1}^y (p_{k,3} w^{-1} + p_{k,4}) > 0.
    }
    The first part of the assertion follows by induction from Lemma \ref{thm:interlacing-w1}.

    Now we turn to the proof of \eqref{eq:Ay+By}. Recall that $G_{y-2} \ll G_{y-1}$ for every $y \geq 2$, and $A_{y-1} \leq A_{y-2}+1$. Assume that $A_{y-1} = A_{y-2}+1$. By Lemma \ref{thm:interlacing-w1}, we have
    \formula{
    A_y &= \max(A_{y-2}+1, A_{y-2} + \ind_{c \neq 0} - \ind_{b=c=0}\} = A_{y-2}+1.
    }
    Since $G_0=G_1=1$ implies that $A_0=A_1=0$, we have $A_y \leq \lfloor y/2 \rfloor$ by induction. An analogous argument applies to $B_y$.
\end{proof}

\begin{lemma}\label{thm:Fyy+10-solution}
    For $y \in \N$ there is a function $H_y \in \mathscr{Q}$ such that
    \formula{
    F_y^{y+1,0}(w) = (p_{y,1}w + p_{y,2}) \, H_y(w).
    }
\end{lemma}

\begin{proof}
    Recall formula \eqref{eq:Fyy+10}:
    \formula{
    F_y^{y+1,0}(w) = \frac{p_{y,1}w+p_{y,2}}{p_{y,3}w^{-1}+p_{y,4}} \, \frac{F_y^{0,y+1}(w)}{F_{y-1}^{0,y}(w)}.
    }
    Lemma \ref{thm:Fy0y+1-solution} implies that there are $G_{y-1}, G_y \in \mathscr{P}$ such that $G_{y-1} \ll G_y$, and the above equation can be rewritten as
    \formula{
    F_y^{y+1,0}(w) = \frac{p_{y,1}w+p_{y,2}}{p_{y,3}w^{-1}+p_{y,4}} \, \frac{G_{y-1}(w)}{G_y(w)} \, \frac{\displaystyle{\prod_{k=1}^y (p_{k,3}w^{-1} + p_{k,4})}}{\displaystyle{\prod_{k=1}^{y-1} (p_{k,3}w^{-1} + p_{k,4})}} = (p_{y,1}w+p_{y,2})\, \frac{G_{y-1}(w)}{G_y(w)}.
    }
    It only remains to denote $H_y = G_{y-1}/G_y$ and observe that $H_y$ belongs to $\mathscr{Q}$ as a quotient of interlacing $\mathscr{P}$ functions.
\end{proof}

By the symmetry of the process, the function $\smash{F_{y+1}^{y,b}}$ that also appears in \eqref{eq:F-explicit-simplified} is completely analogous to $\smash{F_y^{y+1,0}}$. In particular, it satisfies an analogous version of Lemma~\ref{thm:Fyy+10-solution}:
\formula[eq:Fyy-1b]{
F_{y}^{y-1,b}(w) = (p_{y,3}w^{-1}+p_{y,4})\, H_{y,b}(w)
}
for some $H_{y,b} \in \mathscr{Q}$.

In the following result, we find the probability generating function of rescaled first passage locations for our random walks.

\begin{lemma}\label{thm:Fy0b-solution}
    For $0\leq y \leq b$, there are functions $G_y \in \mathscr{P}$ and $H_{y,b} \in \mathscr{Q}$ such that $A_G, B_G$ satisfy \eqref{eq:Ay+By}, and
    \formula{
    F_{y}^{0,b}(w) = \frac{H_{y,b}(w)}{G_{y}(w)} \prod_{k=1}^{y} (p_{k,3} w^{-1} + p_{k,4}).
    }
\end{lemma}

\begin{proof}
    By Lemma \ref{thm:Fy0y+1-solution}, there is a function $G_y \in \mathscr{P}$ with $A_G, B_G$ satisfying \eqref{eq:Ay+By}, such that the right-hand side of \eqref{eq:F-formula} can be rewritten as follows:
    \formula{
    \frac{H_{y,b}(w)}{G_{y}(w)} \prod_{k=1}^{y} (p_{k,3} w^{-1} + p_{k,4}),
    }
    where we denoted
    \formula{
    H_{y,b}(w) = \frac{1}{1- (p_{y,1} w + p_{y,2}) \, F_{y+1}^{y, b}(w) - (p_{y,3} w^{-1} + p_{y,4}) \, F_{y-1}^{y, 0}(w) }\;.
    }
    It remains to prove that the above function belongs to the class $\mathscr{Q}$. By Lemma \ref{thm:Fyy+10-solution} and \eqref{eq:Fyy-1b}, there is a function $\tilde{H}_{y,b} \in \mathscr{Q}$ for which
    \formula{
    (p_{y,1}w + p_{y,2})F_{y+1}^{y,b}(w) &= (p_{y,1}w + p_{y,2})(p_{y+1,3}w^{-1} + p_{y+1,4}) \tilde{H}_{y,b}(w) \\
    &= (aw + b + cw^{-1})  \tilde{H}_{y,b}(w)
    }
    for some $a,b,c \geq 0$, not all equal to zero. Similarly, there is $H_y \in \mathscr{Q}$ such that
    \formula{
    (p_{y,3}w^{-1} + p_{y,4}) F_{y-1}^{y,0}(w) &= (p_{y-1,1}w + p_{y-1,2})(p_{y,3}w^{-1} + p_{y,4}) H_y(w) \\
    &= (a'w + b' + c'w^{-1}) H_y(w)
    }
    for some $a',b',c' \geq 0$, not all equal to zero. Furthermore,
    \formula{
    (p_{y,1}w + p_{y,2})(p_{y+1,3}w^{-1} &+ p_{y+1,4}) \tilde{H}_{y,b}(w) + (p_{y-1,1}w + p_{y-1,2})(p_{y,3}w^{-1} + p_{y,4}) H_y(w) \\
    &= (p_{y,1} + p_{y,2}) \, F_{y+1}^{y, b}(1) + (p_{y,3} + p_{y,4}) \, F_{y-1}^{y, 0}(1) \\
    &= (p_{y,1} + p_{y,2}) \, \P^{(0,y+1)}(\tau_y < \tau_b) + (p_{y,3} + p_{y,4}) \, \P^{(0,y-1)}(\tau_y < \tau_b) \\
    &< (p_{y,1} + p_{y,2}) + (p_{y,3} + p_{y,4}) = 1.
    }
Therefore, we can use Lemma \ref{thm:w2-closure} to find that $H_{y,b} \in \mathscr{Q}$.
\end{proof}

Now, we are in a position to prove the main result of this paper.

\begin{proof}[Proof of Theorem \ref{thm:main}]
    By Lemma \ref{thm:Fy0b-solution},
    \formula{
    F_{y}^{0,b}(w) = \frac{H_{y,b}(w)}{G_{y}(w)} \prod_{k=1}^{y} (p_{k,3} w^{-1} + p_{k,4}),
    }
    where $G_y \in \mathscr{P}$ and $H_{y,b} \in \mathscr{Q}$.
    
    By Lemma \ref{thm:w1-pf}, $(G_y(w))^{-1} \prod_{k=1}^{y} (p_{k,3} w^{-1} + p_{k,4})$ is the generating function of a summable Pólya frequency sequence which is a convolution of at most $\lfloor \frac{1}{2}(y-1) \rfloor$ geometric sequences, at most $\lfloor \frac{1}{2}(y-1) \rfloor$ reflected geometric sequences and at most $y$ mirrored two-point sequences. Lemma \ref{thm:amcm-pf}, in turn, implies that $H_{y,b}$ is the generating function of an $\amcm$ sequence that converges to zero. Furthermore,
    \formula{
    F_y^{0,\infty}(w) = \lim_{b \to \infty} F_{y}^{0,b}(w) = \frac{1}{G_{y}(w)} \prod_{k=1}^{y} (p_{k,3} w^{-1} + p_{k,4}) \lim_{b \to \infty} H_{y,b}(w).
    }

    Recall that a pointwise limit of a sequence of generating functions of $\amcm$ sequences is again the generating function of an $\amcm$ sequence, and the proof is complete.  
\end{proof}

\subsection{Diagonal random walk with barriers}
\label{sec:barriers}

The assumption that our random walk has no natural barriers, that is $p_{y,1} + p_{y,2}>0$ for every $y \in \Z$, allowed us to simplify the proof of Theorem \ref{thm:main}. Now, let us examine the distribution of first passage locations if there is $z \in \Z$ such that $p_{z,1}=p_{z,2}=0$. Probabilistically, this corresponds to a random walk reflected in $y=z$.

Before that, we introduce some auxiliary notation connected to convolution classes of sequences. Recall that $a(k)$ is a reflected geometric sequence if $a(-k)$ is a geometric sequence. An analogous remark applies to two-point sequences. For $y \in \N$, we say that a sequence $a(k)$ belongs to class:
\begin{itemize}
    \item $\mathscr{G}^*_+(y)$ if $a(k)$ is a convolution of at most $y$ geometric sequences;
    \item $\mathscr{G}^*_-(y)$ if $a(k)$ is a convolution of at most $y$ reflected geometric sequences;
    \item $\mathscr{T}_+^*(y)$ if $a(k)$ is a convolution of at most $y$ two-point sequences;
    \item $\mathscr{T}_-^*(y)$ if $a(k)$ is a convolution of at most $y$ reflected two-point sequences.
\end{itemize}
For $y_1, y_2 \geq 0$, we will say that $a(k)$ belongs to $\mathscr{G}^*_+(y_1) * \mathscr{T}_+^*(y_2)$ if $a(k) = (b*c)(k)$, where $b(k) \in \mathscr{G}^*_+(y_1)$ and $c(k) \in \mathscr{T}_+^*(y_2)$. We will use an analogous notation for other classes of sequences.

In terms of the above notation, Theorem \ref{thm:main} translates into the following: assuming that $a=x=0$ and $y \geq 0$, if $a(k)$ is a probability mass function of a scaled first passage location for a random walk that starts from level $y$, then
\formula[eq:ak-convolutions]{
a(k) \in \mathscr{G}^*_+(\lfloor \tfrac{1}{2}(y-1) \rfloor) * \mathscr{G}^*_-(\lfloor \tfrac{1}{2}(y-1) \rfloor) * \mathscr{T}_-^*(y) * \amcm.
}

Now, we turn to the main objective of this section. Recall that by $b \geq 0$ we denoted an additional upper barrier of our random walk that starts from $(0,y)$, see \eqref{eq:generating-function}. Below we consider three mutually exclusive cases.

\textit{Case 1.} If $y < z$, then we can repeat the proof for $b=z+1$ with no major changes.

\textit{Case 2.} If $y=z$, then we can also put $b=z+1=y+1$. By \eqref{eq:F-explicit-simplified}, it follows that
\formula{
    F_y^{0, \infty}(w) = F_{y}^{0,y+1}(w) = \frac{(p_{y,3} w^{-1} + p_{y,4}) F_{y-1}^{0,y}(w)}{1 - (p_{y,3} w^{-1} + p_{y,4}) \, F_{y-1}^{y, 0}(w) }.
}
Note that in this case Lemma \ref{thm:Fy0y+1-solution} holds true, so the above formula describes the generating function of a sequence from class $\mathscr{G}^*_+(\lfloor y/2 \rfloor) * \mathscr{G}^*_-(\lfloor y/2 \rfloor)*\mathscr{T}_-^*(y)$.

If in the above convolution we have at most $\lfloor \frac{1}{2}(y-1) \rfloor$ geometric sequences and at most $\lfloor \frac{1}{2}(y-1) \rfloor$ reflected geometric sequences, it is straightforward that \eqref{eq:ak-convolutions} holds true since the sequence $\ind_{\{0\}}(k)$ is $\amcm$. Hence, we restrict our attention to the case when we have exactly $\lfloor \frac{1}{2}(y-1) \rfloor$ of geometric sequences and exactly $\lfloor \frac{1}{2}(y-1) \rfloor$ reflected geometric sequences.

First, notice that if a sequence $a(k)$ is a convolution of exactly $\lfloor y/2 \rfloor$ geometric sequences, then $a(k) = b(k) * c(k)$, where $b(k)$ is a geometric sequence and $c(k) \in \mathscr{G}^*_+(\lfloor \frac{1}{2}(y-1) \rfloor)$ because we have $\lfloor y/2 \rfloor -1 \leq \lfloor \frac{1}{2}(y-1) \rfloor$. The same remark applies to reflected geometric sequences. Therefore, if $\smash{F_y^{0, y+1}}$ is the generating function of a sequence $a(k)$ which is a convolution of exactly $\lfloor y/2 \rfloor$ geometric sequences, exactly $\lfloor y/2 \rfloor$ reflected geometric sequences and $\mathscr{T}_-^*(y)$ sequence, then $a(k)$ can be represented as a convolution $b(k)*c(k)$, where $b(k) \in \mathscr{G}^*_+(\lfloor \tfrac{1}{2}(y-1) \rfloor) * \mathscr{G}^*_-(\lfloor \tfrac{1}{2}(y-1) \rfloor) * \mathscr{T}_-^*(y)$ and $c(k)$ is a convolution of a geometric sequence $q^k \ind_{\N}(k)$ and a reflected geometric sequence $q^{k} \ind_{\N}(-k)$ with $p,q \in (0,1)$.

Observe that for $|w|=1$ we have
\formula{
\sum_{k=-\infty}^\infty b(k) w^k &= \left(\sum_{k=0}^\infty (qw)^k \right) \left( \sum_{k=0}^\infty (p/w)^k\right) \\
&= \frac{1}{1-qw} \cdot \frac{1}{1-p/w} = \frac{1}{pq} \, \frac{1}{(1/p - 1/w)(1/q-w)}.
}
Clearly, the generating function of $c(k)$ belongs to $\mathscr{Q}$. By Lemma \ref{thm:amcm-pf}, a sequence $c(k)$ is $\amcm$ and \eqref{eq:ak-convolutions} holds true.

\textit{Case 3.} If $y>z$, we use a similar method as in Lemma \eqref{lem:generting-functions-equation}. We split the trajectory of the process with respect to the first visit at level $z$. By the strong Markov property, we deduce that
\formula{
F_y^{0,b}(w) = F_y^{z,b}(w) F_z^{0,b}(w).
}

On one hand, by applying Theorem \ref{thm:main}, we find out that $\smash{F_y^{z,b}}$ is the generating function of a sequence in $\mathscr{G}^*_+(\lfloor \tfrac{1}{2}(y-z-1) \rfloor) * \mathscr{G}^*_-(\lfloor \tfrac{1}{2}(y-z-1) \rfloor) * \mathscr{T}_-^*(y-z) * \amcm$. On the other hand, by applying case 2, we deduce that $\smash{F_z^{0,b}}$ is the generating function of a sequence in $\mathscr{G}^*_+(\lfloor z/2 \rfloor) * \mathscr{G}^*_-(\lfloor z/2 \rfloor)*\mathscr{T}_-^*(z)$. Now \eqref{eq:ak-convolutions} follows from the subadditivity of the floor function.

Finally, we generalize the results to a random walk with multiple barriers. Assume the existence of $n+1$ barriers at $z_0, z_1,\ldots, z_n$ such that $z_0 \geq y > z_1 > \ldots > z_n > 0.$  Then, by an analogous argument, we have
\formula{
F_y^{0,\infty}(w)=F_y^{0,b}(w) = F_y^{z_1,b}(w) F_{z_1}^{z_2,b}(w) \ldots F_{z_{n-1}}^{z_n, b}(w) F_{z_n}^{0,b}(w).
}

By case 1 or 2, we put $b=z_0+1$ and $\smash{F_y^{z_1, b}}$ is the generating function of a sequence from the convolution class $\mathscr{G}^*_+(\lfloor \tfrac{1}{2}(y-z_1-1) \rfloor) * \mathscr{G}^*_-(\lfloor \tfrac{1}{2}(y-z_1-1) \rfloor) * \mathscr{T}_-^*(y-z_1) * \amcm$. Furthermore, by case 3, functions $\smash{F_{z_i}^{z_{i+1},b}}$ for $i=1,2,\ldots,n$, where $z_{n+1}=0$, are the generating functions of $\mathscr{G}^*_+(\lfloor \tfrac{1}{2}(z_i-z_{i+1}) \rfloor) * \mathscr{G}^*_-(\lfloor \tfrac{1}{2}(z_i-z_{i+1}) \rfloor) * \mathscr{T}_-^*(z_i-z_{i+1})$ sequences. Again, by subadditivity, we have the convolution of at most
\formula{
\lfloor \tfrac{1}{2}(y-z_1-1) \rfloor + \sum_{i=1}^{n} \lfloor \tfrac{1}{2}(z_i-z_{i+1}) \rfloor &\leq \lfloor \tfrac{1}{2}(y-z_1-1) \rfloor + \bigg\lfloor \sum_{i=1}^{n} \tfrac{1}{2}(z_i-z_{i+1}) \bigg\rfloor \\
&= \lfloor \tfrac{1}{2}(y-z_1-1) \rfloor + \lfloor \tfrac{1}{2}z_1 \rfloor \leq \lfloor \tfrac{1}{2}(y-1) \rfloor
}
geometric and reflected geometric sequences. Therefore, \eqref{eq:ak-convolutions} holds true.

\begin{remark}
   As discussed in cases 1 and 2, if $y \leq z$, then for $b = z+1$ we have $\smash{F_y^{0,b} = F_y^{0,\infty}}$ and hence, the generating function of an $\amcm$ sequence from Theorem~\ref{thm:main} is given by a $\mathscr{Q}$ function. In other words, the measures $\mu_+$ and $\mu_-$ in representation \eqref{eq:amcm-generating-function} are necessarily purely atomic. 
\end{remark}

\section{Discussion and examples}\label{sec:discussion}

\subsection{Standard $\Z^2$ lattice}\label{sec:news}

Recall that a standard random walk on the lattice $\Z^2$ is a random walk $(X_n, Y_n)$ that makes jumps towards east $(1,0)$, north $(0,1)$, west $(-1,0)$ and south $(0,-1)$, under the assumption that transition probabilities depend only on $Y_n$. Following the notation in the original problem, we denote the transition probabilities as $p_{y,1}$ (east), $p_{y,2}$ (north), $p_{y,3}$ (west), $p_{y,4}$ (south). Furthermore, we assume that $p_{y,2}>0$ and $p_{y,4} > 0$ for every $y \in \Z$. We are again interested in the distribution of the first passage location $X(\tau_a)$. Now, level $a$ can be reached from level $y$ in both odd and even numbers of steps, so no scaling is needed in this case.

The proof of Theorem \ref{thm:news} is generally similar to the one presented in the previous sections. Below we provide its sketch, pointing out the main differences.

\begin{proof}[Sketch of the proof of Theorem \ref{thm:news}]  We divide our argument into four steps.

\textit{Step 1.} The first step of the proof is to establish recurrence relations for the generating functions. The analogue of Lemma \ref{lem:generting-functions-equation} can be proved using the same arguments, and it states that for $a < y< b$ we have
\formula*[eq:news:fyab]{
f_y^{a,b}(z) = \frac{p_{y,4} f_{y-1}^{a,y}(z)}{1-p_{y,1} z- p_{y,2} f_{y+1}^{y,b}(z)- p_{y,3} z^{-1} - p_{y,4} f_{y-1}^{y,a}(z)},
}
where
\formula[eq:f-news]{
f_y^{a,b}(z) = \E^{(0,y)} \left[  z^{X(\tau_a)} \ind_{\tau_a < \tau_b} \right]
}
is the probability generating function of $X(\tau_a)$, provided that the random walk hits $a$ before hitting $b$.

\textit{Step 2.} Now we turn to the algebraic part of this step. By \eqref{eq:news:fyab}, we obtain recurrence equations corresponding to \eqref{eq:system-fy0y+1} and \eqref{eq:system-fyy+10}:
\formula[eq:news:fy0y+1]{
& f_y^{0,y+1}(z) = \frac{p_{y,4} f_{y-1}^{0,y}(z)}{1-p_{y,1}z-p_{y,3}z^{-1} -p_{y,4}f_{y-1}^{y,0}(z)} \quad \text{for } y \geq 1, & f_0^{0,1}(z)=1
}
and
\formula[eq:news:fyy+10]{
& f_y^{y+1,0}(z) = \frac{p_{y,2}}{1-p_{y,1}z-p_{y,3}z^{-1} -p_{y,4}f_{y-1}^{y,0}(z)} \quad \text{for } y \geq 1, & f_0^{1,0}(z)=0.
}
By dividing both sides of \eqref{eq:news:fy0y+1} and \eqref{eq:news:fyy+10}, we get
\formula[eq:news:fyy+10:direct]{
f_{y}^{y+1,0}(z) = \frac{p_{y,2}}{p_{y,4}} \, \frac{f_{y}^{0,y+1}(z)}{f_{y-1}^{0,y}(z)}
}
and after substituting this into \eqref{eq:news:fy0y+1} we have
\formula[eq:news:fy0y+1:2]{
\frac{1}{f_y^{0,y+1}(z)} = \frac{1-p_{y,1}z-p_{y,3}z^{-1}}{p_{y,4}} \, \frac{1}{f_{y-1}^{0,y}(z)} - \frac{p_{y-1,2}}{p_{y-1,4}} \, \frac{1}{f_{y-2}^{0,y-1}(z)}
}
with initial conditions
\formula{
& f_0^{0,1}(z) = 1, & f_1^{0,2}(z) = \frac{p_{1,4}}{1-p_{1,1}z-p_{1,3}z^{-1}}.
}

\textit{Step 3.} In this step we solve \eqref{eq:news:fy0y+1:2} and \eqref{eq:news:fyy+10:direct}. This step involves the classes $\mathscr{P}$ and $\mathscr{Q}$ of rational functions, described in Section \ref{sec:rational-functions}. In order to do this, we need an analogue to Lemma \ref{thm:interlacing-w1} which we formulate below. 

\begin{lemma}\label{thm:news:w1}
Assume that $F \ll G$. For $a,b,c \geq 0$, not all equal to zero, and $d > 0$, denote
    \formula{
    H(x) = (a-bx-cx^{-1}) \, G(x) - d \, F(x).
    }
    If $H(1)>0$, then $H \in \mathscr{P}$ and $G \ll H$. Moreover, we have
    \formula{
    & A_H = \max\{A_G, A_F + \ind_{c \neq 0} - \ind_{a=c=0}\}, & B_H = \max\{ B_G, B_F + \ind_{b \neq 0} - \ind_{a =b=0} \}. &
    }
\end{lemma}

The proof follows the same steps as the proof of Lemma \ref{thm:interlacing-w1}, and so we omit it.

This result allows us to solve \eqref{eq:news:fy0y+1:2} and find out that $g_y = \smash{1/f_y^{0,y+1}}$ are interlacing $\mathscr{P}$ functions. Let us justify this claim. Rewriting \eqref{eq:news:fy0y+1:2} in terms of $g_y$ yields
\formula{
g_y(z) = \left(\frac{1}{p_{y,4}}-\frac{p_{y,1}}{p_{y,4}} z -\frac{p_{y,3}}{p_{y,4}} z^{-1} \right) g_{y-1}(z) - \frac{p_{y-1,2}}{p_{y-1,4}} \, g_{y-2}(z)
}
and
\formula{
& g_0(z) = 1, & g_1(z) = \frac{1}{p_{1,4}}-\frac{p_{1,1}}{p_{1,4}} z -\frac{p_{1,3}}{p_{1,4}} z^{-1}.
}

Clearly, $g_0 \in \mathscr{P}$ as a positive constant function. Note that $d_+(g_1) = \ind_{p_{1,1}\neq0}$ and $d_-(g_1) = \ind_{p_{1,3} \neq 0}$. In each case, we have $d_0(g_1) = d_-(g_1) + d_+(g_1)$ because $p_{1,1}p_{1,3} < 1/4$. By definition, $g_1(1)>0$. This implies that $g_1 \in \mathscr{P}$ and $g_0 \ll g_1$. Our claim follows by induction from Lemma \ref{thm:news:w1}. Moreover, we have $A \leq \lfloor \frac{1}{2}(y+1) \rfloor$ and $B \leq \lfloor \frac{1}{2}(y+1) \rfloor$ in the product representation of $g_y$.

By \eqref{eq:news:fyy+10:direct}, it also follows that $\smash{f_y^{y+1,0}} \in \mathscr{Q}$ as a quotient of interlacing $\mathscr{P}$ functions. Similarly to the original case, one can prove by the symmetry of the random walk that $f_y^{y-1,b} \in \mathscr{Q}$ as well.

\textit{Step 4.} The final step is devoted to investigating the generating functions described by \eqref{eq:news:fyab}, that is
\formula[eq:factors]{
f_y^{0,b}(z) = p_{y,4} f_{y-1}^{0,y}(z) \cdot \frac{1}{1-p_{y,1} z- p_{y,2} f_{y+1}^{y,b}(z)- p_{y,3} z^{-1} - p_{y,4} f_{y-1}^{y,0}(z)}.
}
We denote the first factor as
\formula{
\frac{1}{G_y(z)} = p_{y,4} f_{y-1}^{0,y}(z),
}
where $G_y$ is a $\mathscr{P}$ class function with $A \leq \lfloor y/2 \rfloor$ and $B \leq \lfloor y/2 \rfloor$ in the product representation. The second factor on the right-hand side of \eqref{eq:factors} reads
\formula{
H_{y,b}(z) = \frac{1}{1-p_{y,1} z- p_{y,2} f_{y+1}^{y,b}(z)- p_{y,3} z^{-1} - p_{y,4} f_{y-1}^{y,0}(z)}.
}
Since positive constant functions belong to class $\mathscr{Q}$, we can use Lemma \ref{thm:w2-closure} to justify that $H_{y,b} \in \mathscr{Q}$.

As a result, the generating function of the first passage location can be represented as
\formula{
f_y^{0,b}(z) = \frac{H_{y,b}(z)}{G_y(z)},
}
where $G_y \in \mathscr{P}$ and $H_{y,b} \in \mathscr{Q}$ are described above. The assertion of Theorem \ref{thm:news} now follows from taking limits of both sides of the above equation as $b \to \infty$ and repeating the arguments from the proof of Theorem~\ref{thm:main}.
\end{proof}

\subsection{Honeycomb lattice}
\label{sec:honeycomb}

\begin{figure}
    \centering
    \subfloat[\centering The honeycombs.]{{\includegraphics[width=7cm]{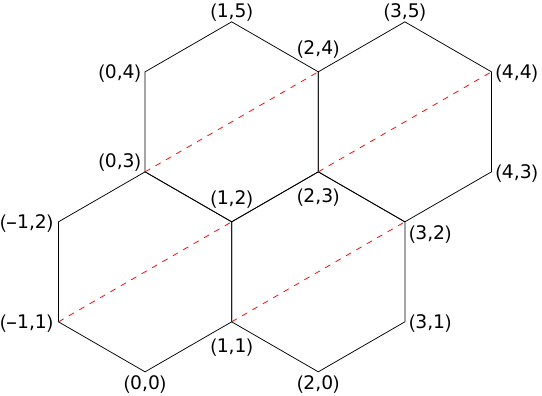}}}
    \qquad
    \subfloat[\centering The $\Z^2$ lattice.]{{\includegraphics[width=7cm]{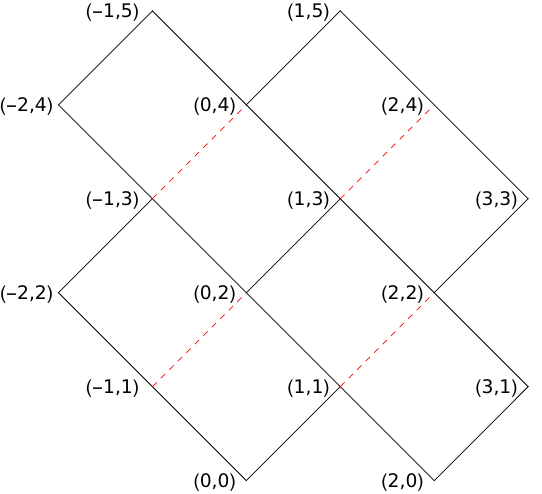}}}
    \caption{Embedding the honeycombs into the $\Z^2$ lattice.}
    \label{fig:honey}
\end{figure}


A random walk on the honeycomb lattice is a two-dimensional random walk $(X_n, Y_n)$ that moves only on the vertices of adjacent regular hexagons. We denote the set of such vertices as 
$\mathscr{H}$. The honeycomb lattice can be naturally represented by the following subset of $\Z^2$:
\formula{
\{(x,y) \in \Z^2: x \operatorname{mod} 1 = 0, \, y \operatorname{mod} 4 \in \{0,3\} \text{ or } x \operatorname{mod} 0 =1, \, y \operatorname{mod} 4 \in \{1,2\}\},
}
as in Figure \ref{fig:honey}A. The process can jump only onto the nearest vertices, meaning that for every point of the honeycomb there are three possible jumps. As in previous cases, we assume that the transition probabilities depend only on the level $Y_n$ of the random walk.

Theorem \ref{thm:main} implies that the rescaled first passage locations have bell-shaped distributions. To see this, we construct an embedding of the honeycomb $\mathscr{H}$ into the diagonal $\Z^2$ lattice. We define $\ph: \mathscr{H} \to \Z^2$ as
\formula{
\ph(x,y) = (x-\lfloor y/2 \rfloor, y).
}

As seen in Figure \ref{fig:honey}, the transformation $\ph$ embeds each vertex of the honeycomb into a vertex of the $\Z^2$ lattice shifted with respect to the first coordinate. In simple words, $\ph$ stretches each honeycomb into a rectangle on the $\Z^2$ lattice, as shown in Figure \ref{fig:honey}B. The red dashed lines symbolize the jumps that cannot be made.

If we assume that the lowest vertices of the honeycomb in Figure \ref{fig:honey}A are at the even level, we can define a two-dimensional random walk on the $\Z^2$ lattice such that
\formula{
& p_{y,1} = 0 \text{ for odd } y, & p_{y,3}=0 \text{ for even } y.
}
and the remaining transition probabilities match those for the honeycomb random walk. Consequently, the distributions of the first passage location of both random walks are the same, up to a shift, and Theorem \ref{thm:main} applies.

\subsection{Explicit examples}\label{sec:examples}

In some special cases, it is easy to find the generating function of the first passage location directly, by solving the recurrence equations. In this section, we denote $\smash{F_y^{0,\infty}}$ simply by $F_y$. Note that if the process starts with level $y$, it can jump up to $y+1$ or down to $y-1$. Translating this into the language of generating functions leads to the following relation:
\formula[eq:ex:Fy]{
F_y(w) = (p_{y,1}w + p_{y,2})F_{y+1}(w) + (p_{y,3}w^{-1}+p_{y,4}) F_{y-1}(w),
}
where $F_0(w)=1$ and $|F_y(w)| \leq 1$ for $|w|=1$.

Below we consider a few special cases of transition probabilities for which the solution is given by an explicit formula. The distribution of the first passage location $\frac{1}{2}(X(\tau_0)-y)$ can be retrieved from $F_y$ by calculating its coefficients in the Laurent series expansion around zero:
\formula{
\P^{(0,y)}(\tfrac{1}{2}(X(\tau_0)-y) = k) = \frac{1}{2\pi} \int_{-\pi}^{\pi} F_y(e^{is}) e^{-iks} \, ds.
}

\begin{example}[Symmetric two-dimensional diagonal random walks]
Let $p_{y,k}=1/4$ for every $y \in \Z$ and $k \in \{1,2,3,4\}$. Then, \eqref{eq:ex:Fy} reads
\formula{
F_y(w) = \frac{(w+1)F_{y+1}(w) + (w^{-1}+1)F_{y-1}(w)}{4},
}
which is a second-order linear recurrence equation. The bounded solution is given by
\formula{
F_y(w) = \left( \frac{2+i(w^{1/2}-w^{-1/2})}{w+1} \right)^y,
}
where $w^{\pm1/2}=e^{\pm iks/2}$ when $w=e^{iks}$, $-\pi \leq s \leq \pi$. The corresponding bell-shaped sequences are presented in Figure \ref{fig:symmetric-2}.


\begin{figure}
    \centering
    \includegraphics[width=\linewidth]{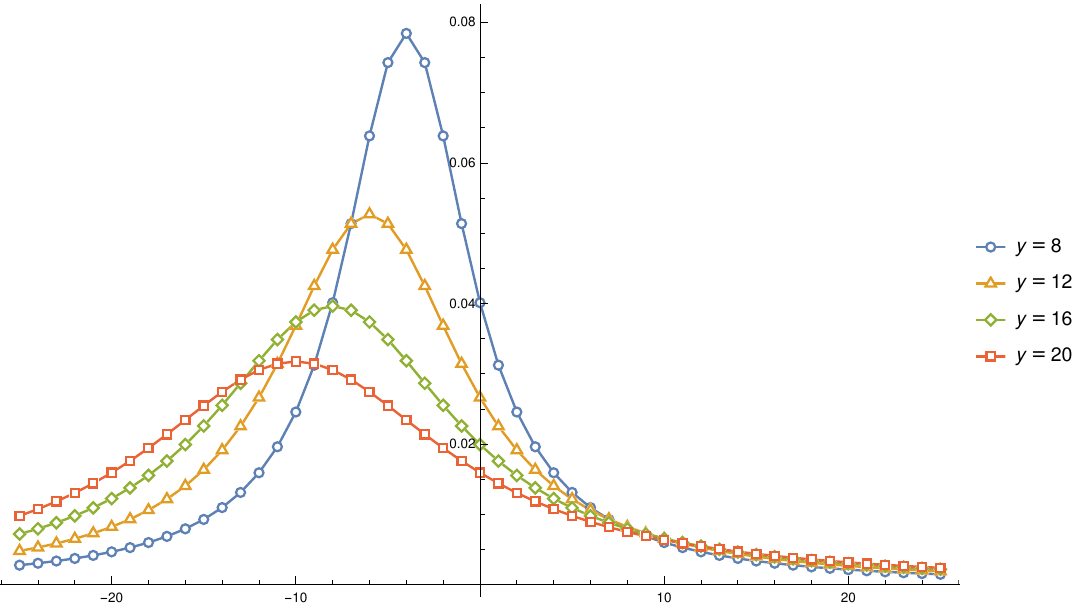}
    \caption{Bell-shaped distribution of $\frac{1}{2}(X(\tau_0)-y)$ for a symmetric two-dimensional diagonal random walk on the lattice $\Z^2$ starting from $(0,y)$.}
    \label{fig:symmetric-2}
\end{figure}

Theorem \ref{thm:main} allows us to claim that the solution given above can be represented in an exponential form, described in detail in Theorem \ref{thm:bs-characterization}(c). The same result can be obtained by studying the complex argument of $F_y$. The detailed calculations were provided in \cite{kw2}; see Example 1.11 and Appendix B.1 therein. In particular, the above means that discrete Poisson kernels are bell-shaped. 
\end{example}

\begin{example}[One-dimensional random walks]\label{ex:bodensson} Bondesson's result is a direct corollary of Theorem \ref{thm:main}. For a given one-dimensional random walk $\tilde{Y}_n$ with jumps only $1$ and $-1$ with probabilities $q_y$ and $1-q_y$ respectively, one can construct the two-dimensional random walk $(X_n, Y_n)=(n, \tilde{Y}_n)$ with transition probabilities
\formula{
& p_{y,1} = q_y = 1-p_{y,4}, & p_{y,2} = p_{y,3} = 0. &
}
It is easy to see that the distribution of $\tilde{\tau}_0 = \min\{n \geq 0: \tilde{Y}_n = 0\}$ with respect to measure $\P^{y}$ is the same as the distribution of $X(\tau_0)=\tau_0$ with respect to measure $\P^{(0,y)}$. 

If $q_y$ do not depend on $y \in \Z$, \eqref{eq:ex:Fy} simplifies to
\formula{
F_y(w) = qw F_{y+1}(w)+ (1-q)F_{y-1}(w).
}
Hence, the bounded solution is given by
\formula{
F_y(w) = \left( \frac{1-\sqrt{1-4q(1-q)w}}{2qw} \right)^y.
}
The bell-shaped probability mass functions of $\frac{1}{2}(\tau_0-y)$ for $q=1/4$ and different $y$ can be seen in Figure \ref{fig:onedimensional-quarter-2}.
\end{example}




\begin{figure}
    \centering
    \includegraphics[width=\linewidth]{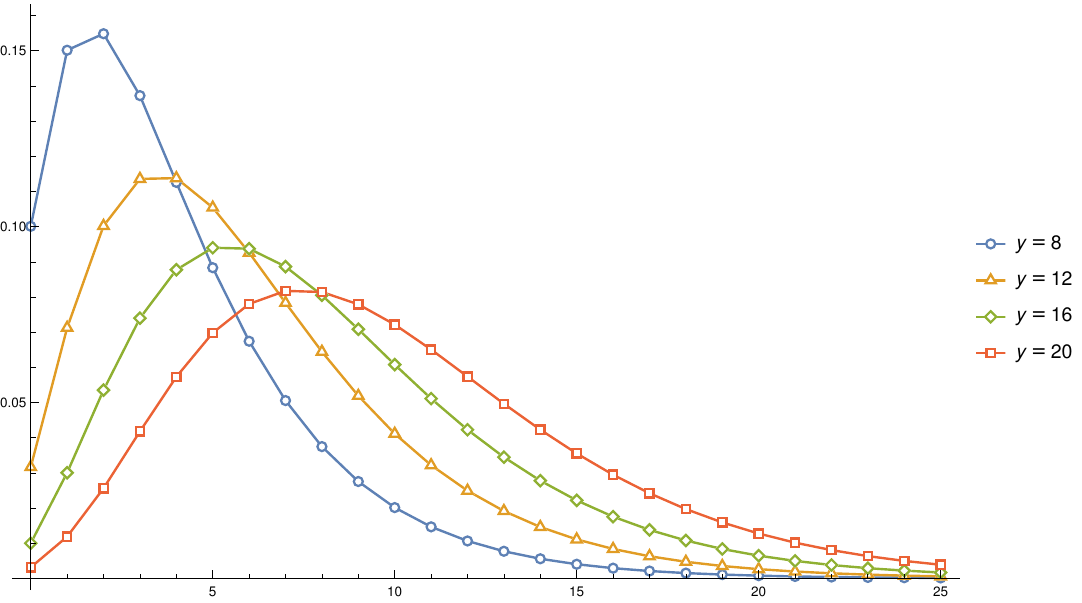}
    \caption{Bell-shaped distribution of $\frac{1}{2}(\tau_0-y)$ for one-dimensional random walk starting from $y$ with $q=1/4$.}
    \label{fig:onedimensional-quarter-2}
\end{figure}

Note that Bondesson's result also implies that the first passage times of two-dimensional random walks follow a bell-shaped distribution. Precisely, given a random walk $(X_n, Y_n)$, we define a one-dimensional random walk $\tilde{Y}_n = Y_n$ with transition probabilities $q_y = p_{y,1} + p_{y,2}$. It is clear that the distributions of $\tau_0$ and $\tilde{\tau}_0$ (with respect to appropriate measures) are identical.

\begin{example}[Symmetric two-dimensional standard random walk]
    Consider a standard random walk on the lattice $\Z^2$ discussed in Section \ref{sec:news}. Let $p_{y,k}=1/4$ for $y \in \Z$ and $k \in \{1,2,3,4\}$. In this case, an analogue of \eqref{eq:ex:Fy} reads
    \formula{
    f_y(z) = \frac{(z+ z^{-1})f_y(z) + f_{y+1}(z) + f_{y-1}(z)}{4},
    }
    where $f_y$ is defined in \eqref{eq:f-news} and $|z|=1$. The bounded solution of this recurrence is given by
    \formula{
    f_y(z) = \left( 2 - \frac{z+z^{-1}}{2} - \sqrt{1-  \frac{z+z^{-1}}{2} } \sqrt{3- \frac{z+z^{-1}}{2} }\right)^y.
    }
    This is the generating function of a bell-shaped sequence presented in Figure \ref{fig:news}. For further discussion, see Example 1.10 and Appendix B.2 in \cite{kw2}.
\end{example}

\begin{figure}
    \centering
    \includegraphics[width=\linewidth]{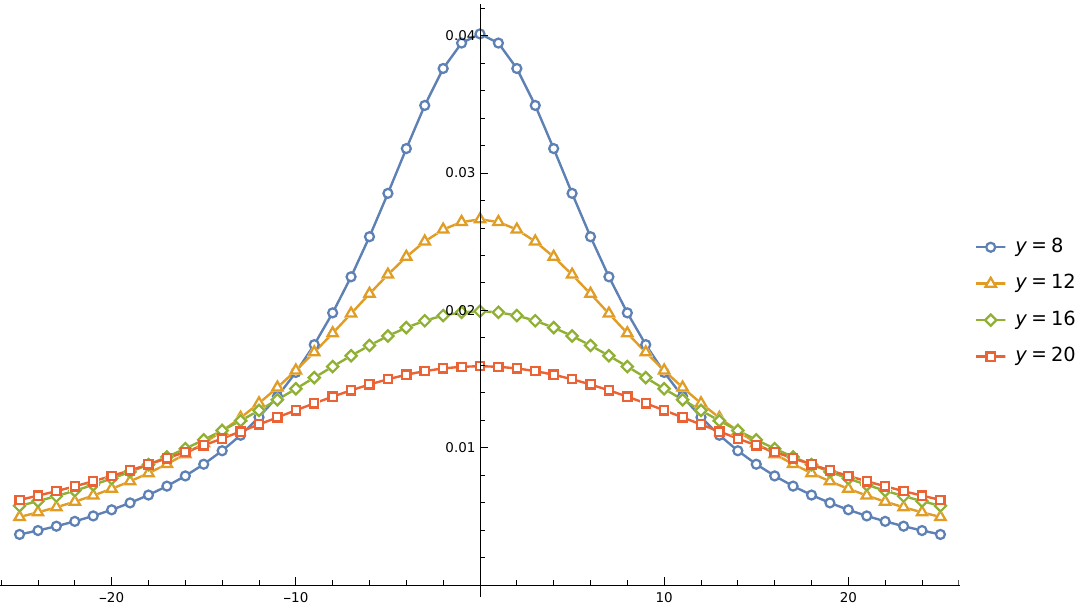}
    \caption{Bell-shaped distribution of $\frac{1}{2}(X(\tau_0)-y)$ for a symmetric two-dimensional standard random walk on the lattice $\Z^2$ starting from $(0,y)$.}
    \label{fig:news}
\end{figure}

\section*{}

\subsection*{Acknowledgments} I would like to express my sincere gratitude to my supervisor, Mateusz Kwaśnicki, for his support throughout the process of writing this paper. His expertise and insightful feedback have greatly contributed to shaping this work.

\bibliographystyle{abbrv}
\bibliography{bibi}

\end{document}